\documentclass[12pt, leqno, two]{article}

\usepackage{CJK}
\usepackage{amsthm}
\usepackage{amsmath,amssymb,mathrsfs}
\usepackage{amsfonts}
\usepackage{yhmath}
\usepackage{graphics}
\usepackage{amssymb}
\usepackage{epsfig}
\usepackage{latexsym}
\usepackage{indentfirst}
\usepackage[top=1in, bottom=1in, left=1.25in, right=1.25in]{geometry}
\numberwithin{equation}{section}
\usepackage{CJK}

\newtheorem{theorem}{Theorem}[section]
\newtheorem{lemma}{Lemma}[section]

\newtheorem{corollary}{Corollary}[section]
\newtheorem{remark}{Remark}[section]
\newtheorem{definition}{Definition}[section]
\numberwithin{equation}{section} 

\title{
Monotonicity formulae, vanishing theorems and some geometric applications
\footnotetext[0]{ ${}^*$Supported by NSFC grant No 10971029, and NSFC-NSF grant No 1081112053.}
}
\author{
Yuxin Dong${}^*$
 and Hezi Lin}
\date{}
\begin{document}

\maketitle

\begin{abstract}
\noindent \textbf{Abstract}. \  Using the stress energy tensor, we establish some monotonicity formulae for
vector bundle-valued $p$-forms satisfying the conservation law, provided that the
base Riemannian (resp. K\"ahler) manifolds poss some real (resp. complex ) $p$-exhaustion functions. Vanishing theorems follow immediately from the monotonicity
formulae under suitable growth conditions on the energy of the $p$-forms.
As an application, we establish a monotonicity formula for the Ricci form of a
K\"ahler manifold of constant scalar curvature and then get a growth condition to
derive the Ricci flatness of the K\"ahler manifold. In particular, when the curvature does not change sign, the K\"ahler manifold is isometrically biholomorphic to $C^m$.
Another application is to deduce
the monotonicity formulae for volumes of minimal submanifolds in some outer
spaces with suitable exhaustion functions. In this way, we recapture the classical
volume monotonicity formulae of minimal submanifolds in Euclidean spaces.
We also apply the vanishing theorems to Bernstein type problem of submanifolds
in Euclidean spaces with parallel mean curvature. In particular, we may obtain
Bernstein type results for minimal submanifolds, especially for minimal real K\"ahler
submanifolds under weaker conditions.
 \end{abstract}

\section{Introduction}
In 1980, Baird and Eells \cite{BE} introduced the stress-energy tensor for maps between Riemannian manifolds,
which unifies various results on harmonic maps. Following \cite{BE}, Sealey \cite{Se} introduced the stress-energy tensor for $p$-forms with
values in vector bundles and established some vanishing theorems for harmonic $p$-forms. Since then, the stress-energy tensors
have become a useful tool for investigating the energy behaviour of vector bundle valued $p$-forms in various problems. Recently the
authors in \cite{DW}  presented a unified method to establish monotonicity formulae
and vanishing theorems for vector-bundled valued $p$-forms satisfying conservation laws
by means of the stress-energy tensors of various energy functionals in geometry and physics.
                         Their method is based on
a fundamental integral formula derived from the stress-energy tensor. Since the stress-energy tensor is a 2-tensor field,
one gets a $1$-form by contracting it with a vector field. The integral formula linked naturally to the conservation law
follows directly from the divergence of the 1-form. In \cite{DW}, the author mainly used the distance function on the base Riemannian manifold
to construct a suitable vector field. The Hessian of the distance function appears naturally in the integral formula. Consequently they
used Hessian comparison theorems and coarea formula to obtain those results.

The similar integral formula technique was also used by \cite{Do,Ta1,Ta2}  to investigate harmonic  maps between K\"ahler
manifolds. Some special exhaustion functions on the domain K\"ahler manifolds were used to establish the monotonicity
formulae of energy by \cite{Ta2} and the monotonicity formulae of the partial energies by \cite{Do} for
the harmonic maps. Consequently some Liouville type results and holomorphicity results were obtained by  \cite{Do} and \cite{Ta2} respectively.
These results are actually equivalent to the vanishing of some $1$-forms with values in the pull-back vector bundles.

In this paper, we shall establish monotonicity formulae and vanishing results for vector bundle valued
$p$-forms satisfying the conservation laws by means of the stress-energy tensors too. Following \cite{Do}, we hope to use some
exhaustion functions instead of distance functions to construct suitable vector fields in the integral formulae
derived by the stress-energy tensors. For these aims, we introduced the concepts of real (resp. complex) $p$-exhaustion
functions on Riemannian manifolds (resp. K\"ahler manifolds) (see their definition in $\S$2). Under some conditions on the
radial curvature of a Riemannian manifold (resp. a K\"ahler manifold), the distance function actually becomes a real
 $p$-exhaustion function (resp. a complex $p$-exhaustion function). The existence of a real $p$-exhaustion function
 on the base manifold  enables us to establish a monotonicity formula for a vector bundle valued $p$-form
 satisfying the conservation law and the growth order of the energy of the $p$-form is determined only by the exhaustion function
 (see Theorem 3.1).
 Inspired by \cite{Do}, we introduce the notion of $J$-invariant $p$-forms, which includes the differential of a pluriconformal
 map, the partial differentials of a map between K\"ahler manifolds,
 the second fundamental of a K\"ahler submanifold and the Ricci form of a K\"ahler manifold,
 etc. It turns out that a complex
 $p$-exhaustion function is more suitable to estimate the growth order of the energy of a $J$-invariant $p$-form.
 As a result, the existence of a complex $p$-exhaustion function enables us to establish a monotonicity formula for a
 $J$-invariant $p$-form satisfying the conservation law too (see Theorem 3.4). Besides establishing the
 global monotonicity formulae, we also consider some conditions which allow us to establish the monotonicity formulae
 outside a compact subset of the base manifolds. Such a kind of monotonicity formulae is useful in some applications.
 Clearly vanishing theorems follow immediately from these monotonicity formulae under suitable growth conditions.

 As an application of the above mentioned results, we establish a monotonicity formula for the Ricci form of a K\"ahler
manifold of constant scalar curvature and then get a growth condition to derive the Ricci flatness of the K\"ahler manifold. Due to the fact that a Ricci form is a $J$-invariant $2$-form, we derive these results on a K\"ahler manifold which posses a complex $2$-exhaustion function.
Furthermore, if the radial curvature of $M$ does not change sign, then the Ricci flatness implies that $M$ is isometrically biholomorphic to $C^m$. In this way, we get some uniformization results.
    On the other hand, if we take $p=0$ and consider a nonzero constant function in the monotonicity formulae, we may get monotonicity formulae for volumes of minimal submanifolds in some outer spaces with suitable exhaustion functions (see Theorem 4.1 and Theorem 4.2). In particular, we recapture the classical volume monotonicity formulae of minimal submanifolds
 in Euclidean spaces.

 The benefits of using more general exhaustion functions instead of the distance functions in establishing monotonicity and vanishing
 results are that one may not only relax the curvature conditions on the domain manifolds but also has more choices for constructing
 suitable vector fields in the integral formula. As an application, we shall investigate Bernstein type problem for
 submanifolds in Euclidean spaces with parallel
     mean curvature. This is still an active project in recent years (see, for examples, \cite{CM,JXY,PRS,RS,SZ1,SZ2,XG,XY1,XY2}
     and the references therein).
Recall that the Gauss maps of these submanifolds are harmonic. Hence the differentials of the Gauss maps satisfy the conservation law
by a result of Baird-Eells \cite{BE}. It is natural to apply the previous vanishing theorems to establish Liouville theorems of the
Gauss maps, which are equivalent to Bernstein type results for the submanifolds. We show that the extrinsic
distance functions of the submanifolds become real $1$-exhaustion functions if the second fundamental forms satisfy some decay conditions.
Consequently we obtain some Bernstein type theorems for submanifolds with parallel mean curvature under certain decay conditions on
the second fundamental forms and growth conditions on the integrals of the squared norms of the second fundamental forms.
      When the submanifolds are minimal, the Bernstein type results may be established under weaker conditions.
Finally we consider the Bernstein problem for minimal submanifolds with finite total scalar curvature. Most results on this
topic are obtained under the assumption that the minimal submanifolds have some stability.
For example, the authors in \cite{SZ1} proved that any stable complete minimal hypersurface with finite total scalar
curvature is a hyperplane. We refer the reader to \cite{FZ,LW,PRS,PV,SZ2,Wa,XD}  for recent progress on this topic
and references therein. According to our technique, we may establish some Bernstein theorems for minimal submanifolds with finite total
scalar curvature by assuming some boundary condition of a compact sublevel set of the extrinsic exhaustion distance function
(see Corollary 5.3, Corollary 5.4 and Remark 5.4), instead of assuming the stable condition. In particular,
we prove that a minimal real K\"ahler submanifolds with finite total scalar curvature must be an affine plane.

\section{Stress energy tensors and exhaustion functions}

     A smooth map $f:M \rightarrow N$ between
       two Riemannian manifolds is said to be a harmonic map if it is a critical point of the following energy functional:\\
     $$E(f)=\frac{1}{2}\int_M |df|^2dv$$
       with respect to any compactly supported variation.

 In \cite{BE}, Baird-Eells introduced the stress-energy tensor $S_f$ associated with $E(f)$ as follows
 $$ S_f=\frac{|df|^2}{2}g-df\odot df,$$ where $df\odot df\in \Gamma(T^*M\otimes T^*M)$ is a symmetric tensor defined by
 $$(df\odot df)(X,Y)=<df(X),df(Y)>.$$ Then they proved that a harmonic map satisfies the conservation law, that is,
 $divS_f=0$.

  Let $(M, g)$  be a Riemannian manifold and $\xi:E\rightarrow M$ be a smooth Riemannian vector bundle
   over $M$ with a metric compatible connection $\nabla^E$. Set $A^p(\xi)=\Gamma (\Lambda^pT^*M \otimes E)$ the space of smooth
   $p$-forms on $M$ with values in the vector bundle $\xi:E\rightarrow M$. The exterior covariant
   differentiation $d^\nabla:A^p(\xi)\rightarrow A^{p+1}(\xi)$ relative to the connection $\nabla^E$ is defined by (cf. \cite{EL})
   \begin{equation*}
   (d^\nabla\omega)(X_1,\cdots,X_{p+1})=\overset{p+1}{\underset{i=1}{\sum}}(-1)^{i+1}(\nabla_{X_i}
   \omega)(X_1,\cdots,\widehat{X_i},\cdots,X_{p+1}).
   \end{equation*}
   The codifferential operator $\delta^\nabla:A^p(\xi)\rightarrow A^{p-1}(\xi)$ characterized as the adjoint of $d^\nabla$ is defined by
 \begin{equation*}
(\delta^{\nabla}\omega)(X_1,\cdots,X_{p-1})=-\underset{i}{\sum}(\nabla_{e_i}\omega)(e_i,X_1,\cdots,X_{p-1}).
\end{equation*}

   For $\omega\in A^p(\xi)$, we define the energy functional of $\omega$ as follows:
   $$E(\omega)=\frac{1}{2}\int_M |\omega|^2dv_g.$$
  Its stress-energy tensor is:
\begin{equation}
S_{\omega} (X,Y)= \frac{|\omega|^2}{2}g(X,Y)-(\omega \odot \omega)(X,Y),
\end{equation}
where $\omega \odot \omega\in \Gamma(A^p(\xi)\otimes A^p(\xi)) $ is a symmetric tensor defined by
$$(\omega \odot \omega)(X,Y)=<i_X\omega,i_Y\omega>.$$
Here $i_X\omega\in A^{p-1}(\xi)$ denotes the interior product by $X\in TM$.
Notice that, if p=0, i.e.,
$\omega\in \Gamma(\xi)$, $i_X\omega=0$ then (2.1) becomes
\begin{equation}
S_{\omega} (X,Y)= \frac{|\omega|^2}{2}g(X,Y).
\end{equation}

For a $2$-tensor field $T\in \Gamma(T^*M\otimes T^*M)$, its divergence $divT\in\Gamma(T^*M)$ is defined by
\begin{equation*}
(divT)(X)=\underset{i}{\sum}(\nabla_{e_i}T)(e_i,X),
\end{equation*}
where $\{e_i\}$ is an orthonormal basis of $TM$. The divergence of $S_{\omega}$ is given by (cf. \cite{Ba,Se,Xi})
\begin{equation}
(divS_\omega)(X)=<\delta^\nabla\omega,i_X\omega>+<i_Xd^\nabla\omega,\omega>.
\end{equation}

For a vector field $X$ on $M$, its dual one form $\theta_X$ is given by
$$\theta_X(Y)=g(X,Y) \qquad \forall X\in TM. $$
The covariant derivative of $\theta_X $ gives a 2-tensor field $\nabla\theta_X$:
$$(\nabla\theta_X)(Y,Z)=(\nabla_Z\theta_X)(Y)=g(\nabla_Z X,Y), \qquad \forall Y,Z\in TM.$$
If $X=\nabla\psi$ is the gradient of some  smooth function $\psi$ on $M$, then $\theta_X=d\psi$ and $\nabla\theta_X=Hess(\psi)$.

For any vector field $X$ on $M$, a direct computation yields (cf. \cite{Xi} or Lemma 2.4 of \cite{DW}):
\begin{equation}
div(i_XS_{\omega})=<S_{\omega},\nabla\theta_X>+(divS_{\omega})(X).
\end{equation}

Let $D$ be any bounded domain of $M$ with $C^1$ boundary. By (2.4) and using the divergence theorem, we immediately have the
following integral formula (see \cite{DW,Xi}):
\begin{equation}
\int_{\partial D}S_{\omega}(X,\nu)ds_g=\int_{D}[<S_{\omega},\nabla \theta_X>+(divS_{\omega})(X)]dv_g.
\end{equation}
where $\nu$ is the unit outward normal vector field along $\partial D$. In particular, if $\omega$ satisfies the conservation
law, i.e. $divS_{\omega}=0$,
 then
\begin{equation}
\int_{\partial D}S_{\omega}(X,\nu)ds_g= \int_{D}<S_{\omega},\nabla \theta_X>dv_g.
\end{equation}

To apply the above integral formula, we introduce some special exhaustion functions.
Let $(M^m, g)$($m\geq 2$) be a $m$-dimensional Riemannian manifold and let $\Phi$ be a Lipschitz continuous function on $M^{m}$
  satisfying the following conditions:

(2.7) $\Phi\geqq 0$ and $\Phi$ is an exhaustion function of $M$, i.e., each sublevel set $B_{\Phi}(t):={{\{\Phi<t}}\}$
 is relatively compact in $M$ for $t\geqq 0$;

 (2.8) $\Psi=\Phi^{2}$ is of class $C^{\infty}$ and $\Psi$ has only discrete critical points;

 (2.9) The constant
  $k_{1}=inf_{x\in M}\{\lambda_{1}(x)+\lambda_{2}(x)+\cdots+\lambda_{m}(x)-2p\lambda_{m}(x)\}>0$, where $p$ is a nonnegative integer,
  $\lambda_1(x)\leq \lambda_2(x)\leq \cdots \leq  \lambda_m(x)$ are the eigenvalues of $Hess(\Psi)$,
  and the constant $k_2=\sup_{x \in M}|\nabla \Phi |^2$ is finite.

When $(M, g)$ is a K\"ahler manifold, we sometimes need the following notion of exhaustion functions, which is defined by the complex Hessian instead of the (real) Hessian of a function.

  Let $\Phi$ be a Lipschitz continuous function on $M^{m} (dim_C M=m>1)$
  satisfying (2.7), (2.8) and the following condition:

(2.10) The constant $\overline{k}_1=\inf_{x\in M}\{\overset{m}{\underset{i=1}{\sum}}\epsilon_i(x)-p\epsilon_m(x)\}$
   is positive where $\epsilon_1 \leq \epsilon_2 \leq  \cdots \leq \epsilon_m $ are the eigenvalues of the complex Hessian
   $ H(\Psi)=(\Psi_{i\overline{j}})$. The constant $\overline{k}_2= \sup_{x\in M}|\nabla \Phi|^2$ is finite. Set
\begin{equation}
 \tag{2.11}
\lambda=\frac{k_1}{2k_2}, \qquad \overline{\lambda}=\frac{\overline{k}_1}{\overline{k}_2}.
\end{equation}
 The function $\Phi$ with the properties (2.7), (2.8) and (2.9) (resp. (2.10)) will be called
  real (resp. complex) $p$-exhaustion function  in this paper. Notice that a complex $1$-exhaustion function is just the special exhaustion function discussed in \cite{Do,Ta2}.

Now we give  some examples of real and complex $p$-exhaustion functions.

\textbf{Example 2.1}.  For  $R^m$, we take$\Psi=\Phi^2=\overset{m}{\underset{i=1}{\sum}}
\frac{(x_i)^2}{a_i}$, where $a_i$ are any positive constants  satisfying
$1\leq a_1 \leq
a_2 \leq \cdots \leq a_m$,
 then $\Phi$ is a smooth exhaustion function of $R^m$, and $\Psi$  has only discrete critical points. For $x\in R^n-\{0\}$, we have
 \begin{equation*}
  \nabla \Phi(x)=\overset{m}{\underset{i=1}{\sum}} \frac{x_i/a_i}{\sqrt{\sum_{i=1}^{m}(x_i)^2/a_i}}\frac{\partial}{\partial x_i},
  \end{equation*}
\begin{equation*}
  Hess(\Psi)(x)= \overset{m}{\underset{i=1}{\sum}}
\frac{2}{a_i}dx_i \bigotimes dx_i.
\end{equation*}
So
\begin{equation*}
k_1=\frac{2}{a_1}+\frac{2}{a_2}+ \cdots +\frac{2}{a_m}-\frac{4p}{a_1},\qquad k_2=\underset{x\in R^m}{\sup}|\nabla\Phi|^2 \leq 1.
\end{equation*}

\textbf{Example 2.2}. For $C^m$, we take $\Psi=\Phi^2=\frac{|z_1|^2}{a_1}+\cdots +\frac{|z_m|^2}{a_m}$, where $a_i$ are any positive constants  satisfying
$1\leq a_1 \leq a_2 \leq \cdots \leq a_m$. A direct computation gives
\begin{equation*}
\overline{k}_1=\frac{2}{a_1}+\frac{2}{a_2}+ \cdots +\frac{2}{a_m}-\frac{2p}{a_1},\qquad \overline{k}_2=
         \underset{x\in C^m}{\sup}|\nabla\Phi|^2 \leq 1.
\end{equation*}
Clearly we may choose suitable $a_i$  such that $k_1$  (resp. $\overline{k}_1$) $>$0.
 Therefore the function $\Phi$ in Example 2.1 (resp. in Example 2.2) becomes a real (resp. complex) $p$-exhaustion function.

 \textbf{Example 2.3}. For a complete submanifold in Euclidean space $R^N$, we shall show in $\S$5 that the restriction of the extrinsic distance
 function becomes a real $1$-exhaustion (or even $p$-exhaustion) function if the second fundamental form of the submanifold satisfies some decay condition (see the proof of Theorem 5.1 or Remark 5.1).

 \textbf{Example 2.4}.  Let $i: M^m \rightarrow C^N$ be an $m$-dimensional submanifold and $F: C^N \rightarrow R$ be a smooth function on $C^N$. Then the composition formula for Hessian of maps gives
 $$Hess(F\circ i)(\eta_j, \overline{\eta}_k)=(HessF)(\eta_j, \overline{\eta}_k)+ dF (B(\eta_j, \overline{\eta}_k)),$$
 where $B$ denotes the second fundamental form of $M$ and
  $\{\eta_j=\frac{1}{\sqrt{2}}(e_j - iJe_j)\}_{j=1}^m$ is any unitary frame tangent to $M$. Since $M$ is a complex submanifold, we have
 $$Hess(F\circ i)(\eta_j, \overline{\eta}_k)=(HessF)(\eta_j, \overline{\eta}_k).$$
 Now let $F=\|z\|^2={\sum}_{A=1}^N|z_A|^2$
 and set $\Psi = \Phi^2=F\circ i$. Clearly the eigenvalues of the complex Hessian $H(\Psi)$ are $\epsilon_1= \cdots = \epsilon_m=2$.
 If $p<m$, then
 $$\overline{k}_1=\underset{x \in M}{\inf}\{\underset{i=1}{\overset{m}{\sum}}\epsilon_i-p \epsilon_m\}=2(m-p)>0.$$
 Therefore $\Phi$ is a complex $p$-exhaustion function with growth order
 $\overline{\lambda}\geq 2(m-p)$, because $\overline{k}_2\leq 1$. Recall that every Stein manifold $M^m$ can be realized as a closed submanifold
 of $C^N$ by a proper holomorphic map $\psi : M^m \rightarrow C^N$. Thus a Stein manifold $M$ admits a complex $p$-exhaustion function $\Psi=\Phi^2$
 with $\Phi = \psi^*(\|z\|)$ and $\overline{\lambda} \geq 2(m-p)$. Notice that the complex Hessian is computed with respect to the induced metric by $\psi$.
 It is well known that every  closed complex submanifold of a Stein manifold is a Stein manifold too. Hence Stein manifolds provide us many examples of K$\ddot{a}$hler manifolds which poss complex $p$-exhaustion functions.

 We now give some complete Riemannian manifolds (resp. K\"ahler manifolds)
whose distance functions are real (resp. complex) $p$-exhaustion functions.
\begin{lemma}
Let $(M^m, g)$($m \geq 2$) be an $m$-dimensional complete Riemannian manifold with a pole $x_0$ and let $r$ be the distance function relative to $x_0$.
Assume that there exist two positive functions $h_1(r)$ and $h_2(r)$ such that
$$h_1(r)[g-dr\otimes dr]\leq Hess(r)\leq h_2(r)[g-dr\otimes dr]$$
in the sence of quadratic forms, then
\begin{equation}
\tag{2.12}
\underset{i=1}{\overset{m}{\sum}}\lambda_i-2p\lambda_m\geq \begin{cases}
2+2(m-1)rh_1(r)- 4prh_2(r) & \text{if $rh_2(r)\geq 1,$} \\
2+ 2(m-1)rh_1(r)-4p  & \text{if  $rh_2(r)< 1,$}
\end{cases}
\end{equation}
\end{lemma}
\noindent
where $\lambda_1\leq \lambda_2\leq \cdots \leq \lambda_m$ are eigenvalues of $Hess(r^2)$. In particular, if $M$ is a
K\"ahler manifold of complex dimension $m$, then
\begin{equation}
\tag{2.13}
\underset{i=1}{\overset{m}{\sum}}\epsilon_i - p\epsilon_m \geq \begin{cases}
1+(2m-1)rh_1(r)-2prh_2(r) & \text{if $rh_1(r)\geq 1,$} \\
1+(2m-1)rh_1(r)- p[1+rh_2(r)] & \text{if $rh_2(r)< 1.$}
\end{cases}
\end{equation}
where $\epsilon_1\leq\epsilon_2\leq \cdots \leq \epsilon_m$ are eigenvalues of the complex Hessian $H(r^2)=((r^2)_{i\overline{j}})$.
\begin{proof}
It is known that $Hess(r^2)$ is given by
$$Hess(r^2)=2dr\otimes dr + 2rHess(r).$$
By the assumption, we have
\begin{equation}
\tag{2.14}
Hess(r^2) \leq 2dr\otimes dr +2rh_2(r)[g-dr\otimes dr].
\end{equation}
If $rh_2(r)\geq 1 (resp. < 1)$, then $\lambda_m\leq 2rh_2(r) (resp. = 2)$. Therefore
\begin{eqnarray*}
\underset{i=1}{\overset{m}{\sum}}\lambda_i-2p\lambda_m &\geq& 2+2(m-1)rh_1(r)- 4prh_2(r)\\
&{}& (resp.\  2+ 2(m-1)rh_1(r)-4p).
\end{eqnarray*}
This proves the first assertion.

Now suppose $M$ is a K\"ahler manifold of complex dimension $m$. Clearly we have (cf. also Lemma 4.4 of \cite{Do})
\begin{equation}
\tag{2.15}
\underset{i=1}{\overset{m}{\sum}}\epsilon_i \geq 1 + (2m-1)rh_1(r).
\end{equation}
By the assumption, we get
\begin{equation}
\tag{2.16}
\epsilon_m \leq  \begin{cases}
2rh_2(r) & \text{if $rh_2(r)\geq 1,$} \\
1+rh_2(r)  & \text{if  $rh_2(r)< 1.$}
\end{cases}
\end{equation}
In conclusion, we derive the following
\begin{equation*}
\underset{i=1}{\overset{m}{\sum}}\epsilon_i - p\epsilon_m \geq \begin{cases}
1+(2m-1)rh_1(r)-2prh_2(r) & \text{if $rh_1(r)\geq 1,$} \\
1+(2m-1)rh_1(r)- p[1+rh_2(r)] & \text{if $rh_2(r)< 1.$}
\end{cases}
\end{equation*}
\end{proof}
\begin{remark}
(1) The authors in \cite{DW} considered the stress energy tensors associated with more general functional,
 that is, the $F$-energy. If we take $F(t)=t$, the proof of Lemma 4.2 in \cite{DW} essentially gives  the lower bounds of
 ${\sum}_{i=1}^{m}\lambda_i-2p\lambda_m$ under the assumption $rh_2(r)\geq 1$.  (2) For $p=1$, it is clear that the estimation
 (2.13) is independent of the upper bound of $Hess(r)$, which was also pointed out in \cite{Do}.
\end{remark}
\begin{lemma}
Let ($M, g)$ be a complete Riemannian manifold with a pole  $x_0$ and let $r$ be the distance function relative to $x_0$.
Denote by $K_r$ the radial curvature of $M$.

(i) If $-\alpha^2\leq K_r \leq -\beta^2$ with $\alpha> 0$, $\beta> 0$, then
$$\beta \coth(\beta r)[g-dr \otimes dr] \leq Hess(r) \leq \alpha \coth(\alpha r)[g-dr \otimes dr].$$

(ii) If $-\frac{A}{(1+r^2)^{1+\epsilon}} \leq K_r \leq \frac{B}{(1+r^2)^{1+\epsilon}}$ with $\epsilon>0$, $A \geq 0$, $0\leq B < 2\epsilon$,
then
$$\frac{1-\frac{B}{2\epsilon}}{r} [g-dr \otimes dr] \leq Hess(r) \leq \frac{e^{\frac{A}{2\epsilon}}}{r}[g-dr \otimes dr]. $$

(iii) If $-\frac{a^2}{1+r^2} \leq K_r \leq \frac{b^2}{1+r^2}$ with $a\geq 0$, $b^2 \in[0,1/4]$, then
$$\frac{1+\sqrt{1-4b^2}}{2r} [g-dr \otimes dr] \leq Hess(r) \leq \frac{1+\sqrt{1+4a^2}}{2r}[g-dr \otimes dr].$$
\end{lemma}
\begin{proof}
The case (i) is standard (cf. \cite{GW}). The case (ii) is discussed in \cite{DW}.
For (iii), we refer the reader to \cite{GW,Kas,PRS}.
\end{proof}

In \cite{DW}, the authors proved that the distance function $r$ of the case (i) (resp. case (ii)) in Lemma 2.2 is a
real $p$-exhaustion function by choosing suitable $\alpha$ and $\beta$ (resp. $A$ and $B$). For the case (iii), that is,
the curvature of $M$ has quadratic decay, we have immediately from (2.12) the following:
\begin{lemma}
Let $(M, g)$ be a complete Riemannian manifold with a pole $x_0$.
If $-\frac{a^2}{1+r^2} \leq K_r \leq \frac{b^2}{1+r^2}$ with $a\geq 0$, $b^2 \in[0,1/4]$ and
 $2+(m-2)(1+\sqrt{1-4b^2})-(2p-1)(1+\sqrt{1+4a^2}) >0$, then $r(x)$  is a real $p$-exhaustion function with
 \begin{equation}
 \tag{2.17}
 \lambda = \frac{2+(m-1)(1+\sqrt{1-4b^2})-2p(1+\sqrt{1+4a^2})}{2}.
 \end{equation}
 \end{lemma}
 \begin{lemma}
Let $M^m$ be a complete K\"ahler manifold of complex dimension $m$ with a pole. Assume that the radial curvature of $M$
satisfies one of the following three conditions:

(i) $-\alpha^2\leq K_r \leq -\beta^2$ with  $\alpha> 0$, $\beta> 0$ and $(2m-1)\beta-2p\alpha > 0$;

(ii) $-\frac{A}{(1+r^2)^{1+\epsilon}} \leq K_r \leq \frac{B}{(1+r^2)^{1+\epsilon}}$ with $\epsilon>0$, $A \geq 0$, $0 \leq B < 2\epsilon$
and $1+ (2m-1)[1-\frac{B}{2\epsilon}]-2pe^{A/2\epsilon} > 0$;

(iii) $-\frac{a^2}{1+r^2} \leq K_r \leq \frac{b^2}{1+r^2}$ with $a\geq 0$, $b^2 \in[0,1/4]$ and
$2+ (2m-1)[1+\sqrt{1-4b^2}]-2p[1+\sqrt{1+4a^2}] > 0$.

Then $r$ is a complex $p$-exhaustion function with growth order
\begin{equation}
\tag{2.18}
\overline{\lambda}= \begin{cases}
2[m-p\frac{\alpha}{\beta}] & \text{if $K_r$ satisfies (i),} \\
1 + (2m-1)(1-\frac{B}{2\epsilon})-2pe^{A/2\epsilon}] & \text{if $K_r$ satisfies (ii),} \\
1 + \frac{(2m-1)[1+\sqrt{1-4b^2}]}{2}-p[1+\sqrt{1+4a^2}] & \text{if $K_r$ satisfies (iii).}
\end{cases}
\end{equation}
\end{lemma}
\begin{proof}
If $K_r$ satisfies (i), then by Lemma 2.2 and the increasing function $\beta r \coth \beta r \rightarrow 1$ as $r\rightarrow 0$,
we meet the case (i) of (2.13). Since $\beta r \coth(\beta r) > 1$ for $r>0$, $\frac{\coth (\alpha r)}{\coth (\beta r)} < 1$
for $0<\beta < \alpha$, we get
\begin{eqnarray*}
\underset{i=1}{\overset{m}{\sum}}\epsilon_i-p\epsilon_m &\geq& 1+ (2m-1)\beta r \coth(\beta r)-2p\alpha r \coth(\alpha r) \\
&=& 1+\beta r\coth(\beta r)[(2m-1)-2p\frac{\alpha r\coth(\alpha r)}{\beta r\coth(\beta r)}]\\
&\geq& 1+(2m-1)- 2p\frac{\alpha}{\beta}\\
&=& 2[m-p\frac{\alpha}{\beta}],
\end{eqnarray*}
provided that $(2m-1)\beta-2p\alpha\geq 0$.

If $K_r$ satisfies (ii) (resp. (iii)), it follows from Lemma 2.2 that we can estimate
$\underset{i=1}{\overset{m}{\sum}}\epsilon_i-p\epsilon_m $ by the second  case of (2.13).
\end{proof}
\begin{remark}
When $M$ is a K\"ahler manifold of complex dimension $m$ (i.e., real dimension $2m$), the distance function $r$ may be both
a real 2$p$-exhaustion function and a complex $p$-exhaustion function under suitable radial curvature conditions listed above. In general,
the growth order $\overline{\lambda}$ is larger than the growth order $\lambda$.
\end{remark}

\section{Monotonicity formulae and Vanishing results}

\begin{theorem}
 Let $(M, g)$ be an $m$-dimensional complete Riemannian manifold with a real $p$-exhaustion function $\Phi$
 and let $\xi:E \rightarrow M$ be a Riemannian vector bundle on $M$.  If $\omega\in A^p(\xi)$ satisfies the conservation law,
that is, $divS_{\omega}=0$, then
\begin{equation}
\frac{1}{\rho_1^ { \lambda}} \int_{B_{\Phi}(\rho_1)}
|\omega|^2dv \leq \frac{1}{\rho_2^{\lambda}}
\int_{B_{\Phi}(\rho_2)}|\omega|^2dv
\end{equation}
for any $0<\rho_1 \leq \rho_2$, where $\lambda$ is given by $\lambda=k_1/2k_2$.
\end{theorem}
\begin{proof}
 Take $\Psi=\Phi^2$, $X=\frac{1}{2}\nabla\Psi$. Obviously $(\nabla \Psi)|_{\partial B_{\Phi}(t)} $ is an outward normal vector field
along $\partial B_{\Phi}(t)$ for a regular value $t>0$ of $\Phi$. Thus $\nabla \Psi= \sigma(x)\nu$ on $\partial B_{\Phi}(t)$ with $\sigma(x)> 0$
 for each point $x\in \partial B_{\Phi}(t)$, where $\nu$ denotes the unit outward normal vector field of $\partial B_{\Phi}(t)$.
Take an orthonormal  basis $\{e_i\}_{i=1,2\cdots m}$ which diagonalize the $Hess(\Psi)$,
then
\begin{eqnarray}
\nonumber <S_\omega,\nabla\theta_X> &=& \frac{1}{2}\overset{m}{\underset{i,j=1}{\sum}}
 S_\omega(e_i,e_j)Hess(\Psi)(e_i,e_j)\\
\nonumber &=& \frac{1}{4} \overset{m}{\underset{i,j=1}{\sum}}|\omega|^2 Hess(\Psi)(e_i,e_j)\delta_{ij}
\end{eqnarray}
\begin{eqnarray}
  \nonumber&&-\frac{1}{2}\overset{m}{\underset{i,j=1}{\sum}}(\omega \odot \omega)(e_i,e_j)Hess(\Psi)(e_i,e_j)\\
\nonumber &\geq& \frac{|\omega|^2}{4}\underset{x\in M}{inf}\{(\lambda_1+\cdots+\lambda_m)-2p\lambda_m\}\\
         &\geq& \frac{1}{4}k_1|\omega|^2.
\end{eqnarray}
By the definition of $S_{\omega}$, we have
\begin{eqnarray}
\nonumber S_{\omega}(X,\nu)&=& \frac{|\omega|^2}{2}g(X,\nu)-(\omega \odot \omega)(X,\nu)\\
 \nonumber &=& \frac{1}{2}r|\omega|^2g(\nabla \Phi,\nu)
-\frac{1}{2}\sigma(x)|i_{\nu}\omega |^2\\
&\leqslant& \frac{\sqrt{k_2}}{2}r|\omega|^2  \qquad  on \ \partial B_{\Phi}(r).
 \end{eqnarray}
Since $\omega $ satisfies the conservation law, then by (2.6), (3.2) and (3.3), we can get
 $$\frac{1}{2}k_1 \int_{ B_\Phi(r)} |\omega|^2dv \leq  \sqrt{k_2}r\int_{\partial B_\Phi(r)} |\omega|^2dv.$$
By the definition of $k_2$ and the co-area formula, we have
\begin{eqnarray*}
r\int_{\partial B_\Phi(r)} |\omega|^2 ds
 &\leq& r\sqrt{k_2}\int_{\partial B_\Phi(r)}\frac{|\omega|^2}{|\nabla \Phi|} ds \\
 &=& r\sqrt{k_2} \frac{d}{dr} \int_{B_\Phi(r)}|\omega|^2dv.
 \end{eqnarray*}
Denote by $\lambda = k_1/2k_2$, then
\begin{eqnarray*}
\lambda \int_{B_\Phi(r)}|\omega|^2dv \leq r\frac{d}{dr} \int_{B_\Phi(r)}|\omega|^2dv,
\end{eqnarray*}
thus
\begin{eqnarray*}
 \frac{\frac{d}{dr} \int_{B_\Phi(r)}|\omega|^2dv}
        {\int_{B_\Phi(r)}|\omega|^2dv}\geq
 \frac{\lambda}{r}.
\end{eqnarray*}
 Integrating the above formula on $[\rho_1,\rho_2]$, we can get $$\frac{1}{\rho_1^ { \lambda}} \int_{B_{\Phi}(\rho_1)}
|\omega|^2dv \leq \frac{1}{\rho_2^{\lambda}}
\int_{B_{\Phi}(\rho_2)}|\omega|^2dv.$$
\end{proof}

Sometimes it is useful to establish monotonicity formulae outside a compact subset on $M$.
Under an extra condition on the $p$-form $\omega$, we may  establish such a kind of monotonicity formulae.
\begin{theorem}
Let $(M, g)$ be an $m$-dimensional complete Riemannian manifold and  let $\xi:E \rightarrow M$ be a Riemannian vector bundle on $M$, $\omega\in A^p(\xi)$. Suppose  $\Phi$ is an exhaustion function on $M$ and is  a real $p$-exhaustion function on $M \setminus B_\Phi (R_0)$ for some $R_0 >0$.
 Set
 \begin{eqnarray*}
 k_1(R_0) &=& inf_{x\in M\setminus B_\Phi(R_0)}\{ \lambda_{1}(x)+\lambda_{2}(x)+\cdots+\lambda_{m}(x)-2p\lambda_{m}(x)\}, \\
 k_2(R_0) &=& \sup_{x \in M\setminus B_\Phi(R_0)}|\nabla \Phi |^2.
 \end{eqnarray*}
 If $\omega$ satisfies the conservation law, and
 \begin{equation}
 \frac{|\omega|^2}{2}\geq |i_{\nu}\omega|^2 \qquad on \ \partial B_\Phi(R_0),
 \end{equation}
 where $\nu$ denotes the unit outward normal vector field of $\partial B_\Phi(R_0)$,
 then
\begin{equation}
 \frac{1}{\rho_1^ { \lambda(R_0)}} \int_{B_{\Phi}(\rho_1)\setminus B_{\Phi}(R_0)}
|\omega|^2dv \leq \frac{1}{\rho_2^{\lambda(R_0)}}
\int_{B_{\Phi}(\rho_2)\setminus B_{\Phi}(R_0)}|\omega|^2dv
\end{equation}
for any $R_0<\rho_1 \leq \rho_2$, where $\lambda(R_0)=k_1(R_0)/2k_2(R_0).$\\
\end{theorem}

\begin{proof}
Take $X=\frac{1}{2}\nabla \Phi^2=\Phi \nabla \Phi$. By  the  definition of $S_\omega$ and (3.4), we get
\begin{eqnarray*}
S_\omega(\nu,\nu)&=& \frac{|\omega|^2}{2}g(\nu,\nu)-(\omega \odot \omega)(\nu,\nu)\\
  &=& \frac{1}{2}|\omega|^2-
|i_{\nu}\omega |^2\\
&\geq& 0.
\end{eqnarray*}
For any $r>R_0$, we set $D=B_\Phi(r)\backslash B_\Phi(R_0)$,
by applying the integral formula (2.6) on $D$ and using (3.2), (3.3), we have
\begin{eqnarray}
\nonumber \frac{1}{4}k_1(R_0) \int_{ B_\Phi(r)\setminus B_\Phi(R_0)} |\omega|^2dv &\leq& \int_{\partial B_\Phi(r)}S_\omega(X,\nu)ds-
\int_{\partial B_\Phi(R_0)}S_\omega(X,\nu)ds\\
\nonumber &=& \int_{\partial B_\Phi(r)}S_\omega(X,\nu)ds-R_0\int_{\partial B_\Phi(R_0)}S_\omega(\nu,\nu)ds\\
&\leq&  \frac{\sqrt{k_2(R_0)}}{2}r\int_{\partial B_\Phi(r)} |\omega|^2ds,
\end{eqnarray}
and by the co-area formula
\begin{eqnarray}
\nonumber r\int_{\partial B_\Phi(r)} |\omega|^2ds &\leq& r\sqrt{k_2{(R_0)}}\int_{\partial B_\Phi(r)}\frac{|\omega|^2}{|\nabla \Phi|} ds\\
 &\leq& r\sqrt{k_2(R_0)} \frac{d}{dr} \int_{B_\Phi(r)\setminus B_\Phi(R_0)}|\omega|^2dv.
 \end{eqnarray}
 Then (3.6) and (3.7) yield
 \begin{equation*}
 \lambda(R_0) \int_{ B_\Phi(r)\setminus B_\Phi(R_0)} |\omega|^2dv \leq r\frac{d}{dr} \int_{B_\Phi(r)\setminus B_\Phi(R_0)}|\omega|^2dv,
 \end{equation*}
 which implies that
 \begin{equation}
  \frac{d}{dr}\{r^{-\lambda(R_0)} \int_{B_\Phi(r)\setminus B_\Phi(R_0)}|\omega|^2dv\}\geq 0.
 \end{equation}
By integrating (3.8) on $[\rho_1,\rho_2]$, we can get the formula (3.5).
\end{proof}

Suppose $M$ is a submanifold described in Example 2.3 on page 7. Then we may establish the monotonicity formulae for forms on $M$. Next, Lemma 2.3 and Theorem 3.1 imply the following:
\begin{theorem}
Let $(M, g)$ be an $m$-dimensional complete Riemannian manifold with a pole $x_0$.
 Let $\xi:E \rightarrow M$ be a Riemannian vector bundle on $M$ and  $\omega\in A^p(\xi)$. Assume that the radial curvature
$-\frac{a^2}{1+r^2} \leq K_r \leq \frac{b^2}{1+r^2}$ with $a\geq 0$, $b^2 \in[0,1/4]$ and
$2+(m-1)(1+\sqrt{1-4b^2})-2p(1+\sqrt{1+4a^2}) >0$.
If $\omega$ satisfies the conservation law, then
 \begin{equation*}
\frac{1}{\rho_1^ { \lambda}} \int_{B_{\rho_1}(x_0)}
|\omega|^2dv \leq \frac{1}{\rho_2^{\lambda}}
\int_{B_{\rho_2}(x_0)}|\omega|^2dv
\end{equation*}
for any $0< \rho_1\leq \rho_2$, where $\lambda$ is given by
\begin{equation}
 \lambda =\frac{2+(m-1)(1+\sqrt{1-4b^2})-2p(1+\sqrt{1+4a^2})}{2}.
 \end{equation}
\end{theorem}

\begin{definition}
A $p$-form $\omega \in A^p(\xi)$ is called $J$-invariant (or pluriconformal) if $\omega$ satisfies :\\
   $$ (\omega  \odot \omega)(JX,JY) =(\omega  \odot \omega)(X,Y)$$
   for $\forall X,Y \in T_xM$.\\
\end{definition}
Now we consider $J$-invariant $p$-forms on K\"ahler manifolds. The work in \cite{Do} suggests that a complex $p$-exhaustion function is suitable
to estimate the growth order of a $J$-invariant $p$-form which satisfies the conservation law.
 \begin{theorem}
  Let $M$ be a complete K\"ahler manifold which posses a complex $p$-exhaustion function $\Phi$.
  Let $\xi: E \rightarrow M$ be a Riemannian vector bundle on $M$. If $\omega \in A^p(\xi)$
  is $J$-invariant and satisfies the conservation law, then
   $$\frac{1}{\rho_1^{\overline{\lambda}}}\int_{B_\Phi (\rho_1)}|\omega|^2dv \leq
   \frac{1}{\rho_2^{\overline{\lambda}}}\int_{B_\Phi (\rho_2)}|\omega|^2dv$$
   for any $0<\rho_1 \leq \rho_2$, where $\overline{\lambda}=\overline{k}_1/\overline{k}_2$.
\end{theorem}
\begin{proof}
Take $X=\frac{1}{2}\nabla \Phi^2=\Phi \nabla \Phi$.
Obviously $(\nabla \Psi)|_{\partial B_{\Phi}(t)} $ is an outward normal vector field
along $\partial B_{\Phi}(t)$ for a regular value $t>0$ of $\Phi$. Thus $\nabla \Psi= \psi(x)\nu$ on $\partial B_{\Phi}(t)$ with $\psi(x)>0$ for
each point $x\in \partial B_{\Phi}(t)$, where $\nu$ denotes the unit outward normal vector field of $\partial B_{\Phi}(t)$.
\begin{eqnarray}
\nonumber S_\omega (X,\nu)&=& \frac{|\omega|^2}{2}<X,\nu>_g -<i_X \omega,i_\nu \omega>_g\\
\nonumber &=& t\frac{|\omega|^2}{2}<\nabla \Phi,\nu>_g -\psi |i_\nu \omega|^2_g\\
   &\leq& \frac{t\sqrt{\overline{k}_2}}{2}|\omega|^2
\end{eqnarray}
on $ \partial B_{\Phi}(t)$ and
\begin{eqnarray}
\nonumber <S_\omega,\nabla \theta_X> &=& \frac{1}{2}<S_\omega,Hess(\Psi)>\\
   &=& \frac{1}{2}[\frac{|\omega|^2}{2}\Delta_g \Psi - <\omega \odot \omega,Hess(\Psi)>],
\end{eqnarray}
  where $\Delta_g$ denote the Laplace-Beltrami operator on $M$. We choose a unitary basis $\{\eta_i=(e_i-iJe_i)/\sqrt{2}\}_{i=1,\cdots,m}$
  at a point $p\in  B_\Phi(t)$ such that
  $$Hess(\Psi)(\eta_i,\overline{\eta}_j(p))=\epsilon_i\delta_{ij},$$
  which is equivalent to
 \begin{eqnarray}
\nonumber Hess(\Psi)(e_i,e_j)+Hess(\Psi)(Je_i,Je_j)&=& 2\epsilon_i\delta_{ij},\\
 Hess(\Psi)(e_i,Je_j)-Hess(\Psi)(Je_i,e_j)&=&0.
 \end{eqnarray}
  Obviously $\{e_i,Je_i\}_{i=1,\cdots,m}$ is an orthonormal basis. Then
\begin{equation}
\Delta_g\Psi=2\overset{m}{\underset{i=1}{\sum}}\epsilon_i.
\end{equation}
Since $\omega$ is $J$-invariant, we have
  \begin{eqnarray}
 && \langle \omega \odot \omega,Hess(\Psi) \rangle\\
\nonumber&=& \underset{i,j}{\sum} [\langle i_{e_i}\omega,i_{e_j}\omega \rangle Hess(\Psi)(e_i,e_j)
               + \langle i_{Je_i}\omega,i_{Je_j}\omega \rangle Hess(\Psi)(Je_i,Je_j)\\
\nonumber && + \langle i_{e_i}\omega,i_{Je_j}\omega \rangle Hess(\Psi)(e_i,Je_j)
             + \langle i_{Je_i}\omega,i_{e_j}\omega \rangle Hess(\Psi)(Je_i,e_j)]\\
\nonumber &=& \underset{i,j}{\sum} \langle i_{e_i}\omega,i_{e_j}\omega \rangle [Hess(\Psi)(e_i,e_j)+ Hess(\Psi)(Je_i,Je_j)]\\
\nonumber &&+ \langle i_{e_i}\omega,i_{Je_j}\omega \rangle [Hess(\Psi)(e_i,Je_j)-Hess(\Psi)(Je_i,e_j)]\\
\nonumber &=& 2\underset{i,j}{\sum}|i_{e_i}\omega|^2\epsilon_i.
\end{eqnarray}
From (3.11), (3.13) and (3.14), we obtain
\begin{eqnarray}
\nonumber<S_\omega,\nabla\theta_X>&=& \frac{1}{2} \left[\frac{|\omega|^2}{2}(2\overset{m}{\underset{i=1}{\sum}}\epsilon_i)-
     2\overset{m}{\underset{i=1}{\sum}}|i_{e_i}\omega|^2 \epsilon_i\right]\\
\nonumber &\geq& \frac{1}{2}\left[|\omega|^2 \overset{m}{\underset{i=1}{\sum}}\epsilon_i-
     2\overset{m}{\underset{i=1}{\sum}}|i_{e_i}\omega|^2 \epsilon_m \right]\\
\nonumber &=&  \frac{1}{2}\left[|\omega|^2 \overset{m}{\underset{i=1}{\sum}}\epsilon_i-p|\omega|^2 \epsilon_m\right]\\
\nonumber &=& \frac{1}{2}|\omega|^2 (\overset{m}{\underset{i=1}{\sum}}\epsilon_i-p\epsilon_m)\\
&\geq& \frac{1}{2}\overline{k}_1|\omega|^2,
\end{eqnarray}
where we have used the fact that $2\sum_i|i_{e_i}\omega|^2=p|\omega|^2$.\\
 Since $divS_\omega=0$, it follows from (2.6), (3.10) and (3.15) that
 \begin{equation*}
 \frac{\sqrt{\overline{k}_2}}{2}t\int_{\partial B_\Phi(t)}|\omega|^2 ds \geq \frac{\overline{k}_1}{2}\int_{B_\Phi(t)}|\omega|^2dv,
 \end{equation*}
and  the coarea formula  yields
\begin{equation*}
\frac{d}{dt}\int_{B_\Phi(t)}|\omega|^2 dv  \geq \frac{1}{\sqrt{\overline{k}_2}}\int_{\partial B_\Phi(t)}|\omega|^2 dv.
\end{equation*}
  Hence
 \begin{equation}
 \frac{\frac{d}{dt}\int_{B_\Phi(t)}|\omega|^2 dv}{\int_{B_\Phi(t)}|\omega|^2dv}\geq \frac{\overline{\lambda}}{t}.
 \end{equation}
 This theorem  follows immediately from integrating (3.16) on $[\rho_1, \rho_2]$.
\end{proof}

\begin{theorem}
 Let $M$ be a complete K\"ahler manifold. Suppose  $\Phi$ is an exhaustion function on $M$ and is  a complex $1$-exhaustion function on $M \setminus B_\Phi (R_0)$ for some $R_0 >0$.
 Set
 \begin{eqnarray*}
 \overline{k_1}(R_0) &=& \underset{x\in M\setminus B_\Phi(R_0)}{inf}\{\overset{m}{\underset{i=1}{\sum}}\epsilon_i(x)-\epsilon_m(x)\}, \\
 \overline{k_2}(R_0) &=& \sup_{x \in M\setminus B_\Phi(R_0)}|\nabla \Phi |^2.
 \end{eqnarray*}
   If $\omega \in A^1(\xi)$ is $J$-invariant and satisfies the conservation law, then
   $$\frac{1}{\rho_1^{\overline{\lambda}(R_0)}}\int_{B_\Phi (\rho_1)\setminus B_\Phi (R_0)}|\omega|^2 dv \leq
   \frac{1}{\rho_2^{\overline{\lambda}(R_0)}}\int_{B_\Phi (\rho_2)\setminus B_\Phi (R_0)}|\omega|^2 dv$$
   for any $R_0 <\rho_1 \leq \rho_2$, where $\overline{\lambda}(R_0)=\overline{k}_1(R_0)/\overline{k}_2(R_0).$
\end{theorem}
\begin{proof}
Take $\Psi=\Phi^2,X=\frac{1}{2}\nabla\Psi$.
For $\omega \in A^1(\xi)$, by the $J$-invariant of $\omega$, we have
\begin{eqnarray}
\nonumber S_\omega(\nu,\nu)&=& \frac{|\omega|^2}{2}g(\nu,\nu)-(\omega \odot \omega)(\nu,\nu)\\
 \nonumber  &\geq& \frac{1}{2}[<\omega(\nu),\omega(\nu)>+<\omega(J\nu),\omega(J\nu)>]-<\omega(\nu),\omega(\nu)>\\
&\geq& 0.
\end{eqnarray}
For any $R>R_0$, we set $D=B_\Phi(R)\backslash B_\Phi(R_0)$,
by applying the integral formula(2.6) on $D$ and arguing
in a similar way as in Theorem 3.2 and Theorem 3.4, considering (3.17), we can get the proof of this theorem.
\end{proof}

From Example 2.4 and Theorem 3.4, we immediately have the following.
\begin{theorem}
Let $M^m$  be a complex $m$-dimensional Stein manifold with metric and exhaustion function $\Phi$  as in Example 2.4. Suppose $\xi: E \rightarrow M$ is a Riemannian vector bundle
on $M$ and $\omega \in A^p(\xi)$ with $p<m$ is $J$-invariant. If $\omega$ satisfies the conservation law, then
\begin{equation*}
\frac{1}{\rho_1^{2(m-p)}}\int_{B_{\Phi}(\rho_1)}|\omega|^2 dv \leq \frac{1}{\rho_2^{2(m-p)}}\int_{B_{\Phi}(\rho_2)}|\omega|^2 dv
\end{equation*}
for any $0 < \rho_1 \leq \rho_2$.
\end{theorem}
\begin{theorem}
  Let $M^m$ $(dim_C M =m)$ be a complete K\"ahler manifold with a pole $x_0$. Let $\xi: E \rightarrow M$ be a Riemannian vector bundle on $M$.
  Assume that the radial curvature $K_r$ of $M$ satisfies
  one of the following three conditions:

  (i) $-\alpha^2\leq K_r \leq -\beta^2$ with  $\alpha> 0$, $\beta> 0$ and $(2m-1)\beta-2p\alpha \geq 0$;

(ii) $-\frac{A}{(1+r^2)^{1+\epsilon}} \leq K_r \leq \frac{B}{(1+r^2)^{1+\epsilon}}$ with $\epsilon>0$, $A \geq 0$, $0 \leq B < 2\epsilon$
and $1+ (2m-1)[1-\frac{B}{2\epsilon}]-2pe^{A/2\epsilon} > 0$;

(iii) $-\frac{a^2}{1+r^2} \leq K_r \leq \frac{b^2}{1+r^2}$ with $a\geq 0$, $b^2 \in[0,1/4]$ and
$2+ (2m-1)[1+\sqrt{1-4b^2}]-2p[1+\sqrt{1+4a^2}] > 0$.

   If $\omega \in A^p(\xi)$ is $J$-invariant and satisfies the conservation law, then
   $$\frac{1}{\rho_1^{\overline{\lambda}}}\int_{B_{\rho_1}(x_0)}|\omega|^2 dv \leq \frac{1}{\rho_2^{\overline{\lambda}}}
   \int_{B_{\rho_2}(x_0)}|\omega|^2 dv$$
   for any $0<\rho_1 \leq \rho_2$, where
\begin{equation}
\overline{\lambda}= \begin{cases}
2[m-p\frac{\alpha}{\beta}] & \text{if $K_r$ satisfies (i),} \\
1 + (2m-1)(1-\frac{B}{2\epsilon})-2pe^{A/2\epsilon}] & \text{if $K_r$ satisfies (ii),} \\
1 + \frac{(2m-1)[1+\sqrt{1-4b^2}]}{2}-p[1+\sqrt{1+4a^2}] & \text{if $K_r$ satisfies (iii).}
\end{cases}
\end{equation}
\end{theorem}
\begin{proof}
It is a direct proof by Lemma 2.4 and Theorem 3.4.
\end{proof}

Now we give some concrete examples of $J$-invariant forms which satisfy the
conservation law.

\textbf{Example 3.1}.
 Let $f : (M, g, J)\rightarrow N$ be a smooth map from a K\"ahler manifold. We say
that $f$ is pluriconformal if $< df(JX), df(JY)>=< df(X), df(Y)>$. Hence the
differential $df$ of a pluriconformal harmonic map is a $J$-invariant $1$-form with
value in $f^{-1}TN$ which satisfies the conservation law.

\textbf{Example 3.2}.  Let $f : (M, g, J) \rightarrow (N, h, J')$ be a smooth map between two K\"ahler
manifolds. Define two $1$-forms $\sigma$, $\tau \in A^1(f^{-1}TN)$
\begin{equation*}
\sigma(X) = \frac{df(X)+J'df(JX)}{2}, \qquad  \tau(X) = \frac{df(X)- J'df(JX)}{2}
\end{equation*}
for any $X\in TM$. Then $\sigma$ and $\tau$ are $J$-invariant. Furthermore, $\sigma$ and $\tau$  satisfy
the conservation laws if $f$ is pluriharmonic map (cf. \cite{Do}).

\textbf{Example 3.3}.  Let $\varphi: M \rightarrow N$ be a complex submanifold of a K\"ahler manifold.
It is known that the second fundamental form $B$ of $M$ satisfies
\begin{equation}
B(JX, Y) = JB(X, Y).
\end{equation}
If we regard $B \in A^1(T^*M \otimes \varphi^{-1}N)$, then (3.19) implies that $B$ is a $J$-invariant
$1$-form with value in $T^*M \otimes \varphi^{-1}N$. When $N$ is a complex space form, $B$ satisfies
the following Codazzi equation
\begin{equation}
(\nabla_XB)(Y,Z) = (\nabla_Y B)(X,Z).
\end{equation}
Then (3.20) is equivalent to $d^\nabla B = 0$. Since a complex submanifold is automatically
minimal, (3.20) yields $\delta^\nabla B = 0$ too. It follows from (2.3) that $B$ satisfies the
conservation law.

\textbf{Example 3.4}.  The Ricci form $\Omega_{ric}$ of a K\"ahler manifold $(M^m, g, J)$ is defined by $\Omega_{ric}(X, Y ) =Ric(JX, Y)$.
 A basic property of $\Omega_{ric}$ is that $\Omega_{ric}(JX,JY) =\Omega_{ric}(X, Y)$. So $\Omega_{ric}$
is a $J$-invariant $2$-form in the sense of Definition 3.1. On the other hand, we
know that $\Omega_{ric}$ is closed and
\begin{equation}
\delta \Omega_{ric}=-\frac{1}{2}JdS,
\end{equation}
where $S$ denotes the scalar curvature of $M$ (cf. \cite{Mor}). By (3.21) and (2.3), we
see that if $M^m$ has constant scalar curvature, then its Ricci form $\Omega_{ric}$ satisfies the
conservation law.

In \cite{Do}, the author established some monotonicity formulae for the energy or
partial energies of harmonic maps, pluriconformal harmonic maps and pluriharmonic
maps by considering $df$, $\sigma$ and $\tau$. Consequently some Liouville results and
holomorphicity results were obtained.

In $\S$5, we shall consider submanifolds in $R^N$ with parallel mean curvature,
including minimal submanifolds and minimal real K\"ahler submanifolds. The so-called
minimal real K\"ahlerr submanifolds are generalizations of K\"ahler submanifolds.

From Theorem 3.4 and Example 3.4, we immediately get
\begin{theorem}
Let $(M^m, g, J)$ be a complex $m$-dimensional K\"ahler manifold
which posses a complex $2$-exhaustion function $\Phi$. If $M$ has constant scalar curvature,
then
\begin{equation*}
\frac{1}{\rho_1^{\overline{\lambda}}}\int_{B_{\Phi}(\rho_1)}|\Omega_{ric}|^2 dv \leq \frac{1}{\rho_2^{\overline{\lambda}}}\int_{B_{\Phi}(\rho_2)}|\Omega_{ric}|^2 dv
\end{equation*}
for any $0 < \rho_1 \leq \rho_2$, where $\overline{\lambda}=\overline{k}_1/ \overline{k}_2$.
\end{theorem}

By Example 2.4 and Theorem 3.8, we have
\begin{corollary}
Let $M^m$ $(m > 2)$ be a complex $m$-dimensional Stein manifold with metric and
$\Phi$  as in Example 2.4. If $M$ has constant scalar curvature, then
\begin{equation*}
\frac{1}{\rho_1^{2(m-2)}}\int_{B_{\Phi}(\rho_1)}|\Omega_{ric}|^2 dv \leq \frac{1}{\rho_2^{2(m-2)}}\int_{B_{\Phi}(\rho_2)}|\Omega_{ric}|^2 dv
\end{equation*}
for any $0 < \rho_1 \leq \rho_2$.
\end{corollary}
Taking $p = 2$ in Lemma 2.4, and using Theorem 3.8, we yield
\begin{corollary}
Let $M^m$ $(dim_C M =m)$ be a complete K\"ahler manifold with a pole. Assume
that the radial curvature of $M$ satisfies one of the following three conditions:

(i) $-\alpha^2\leq K_r \leq -\beta^2$ with  $\alpha> 0$, $\beta> 0$ and $m\beta-2\alpha \geq 0$;

(ii) $-\frac{A}{(1+r^2)^{1+\epsilon}} \leq K_r \leq \frac{B}{(1+r^2)^{1+\epsilon}}$ with $\epsilon>0$, $A \geq 0$, $0 \leq B < 2\epsilon$
and $1+ (2m-1)[1-\frac{B}{2\epsilon}]-4e^{A/2\epsilon} > 0$;

(iii) $-\frac{a^2}{1+r^2} \leq K_r \leq \frac{b^2}{1+r^2}$ with $a\geq 0$, $b^2 \in[0,1/4]$ and
$2+ (2m-1)[1+\sqrt{1-4b^2}]-4[1+\sqrt{1+4a^2}] > 0$.

If $M$ has constant scalar curvature, then
\begin{equation*}
\frac{1}{\rho_1^{\overline{\lambda}}}\int_{B_{\Phi}(\rho_1)}|\Omega_{ric}|^2 dv \leq \frac{1}{\rho_2^{\overline{\lambda}}}\int_{B_{\Phi}(\rho_2)}|\Omega_{ric}|^2 dv
\end{equation*}
for any $0 < \rho_1 \leq \rho_2$, where
\begin{equation}
\overline{\lambda}= \begin{cases}
2[m-2\frac{\alpha}{\beta}] & \text{if $K_r$ satisfies (i),} \\
1 + (2m-1)(1-\frac{B}{2\epsilon})-4e^{A/2\epsilon}] & \text{if $K_r$ satisfies (ii),} \\
1 + \frac{(2m-1)[1+\sqrt{1-4b^2}]}{2}-2[1+\sqrt{1+4a^2}] & \text{if $K_r$ satisfies (iii).}
\end{cases}
\end{equation}
\end{corollary}

Now we deduce some vanishing theorems from the above monotonicity formulae. First, by Theorem 3.1 and Theorem 3.4, we have:
\begin{corollary}
Let $M^m$, $\xi:E \rightarrow M$, $\Phi$  and $\omega$ be either as in Theorem 3.1 or as in Theorem 3.4.  If $\omega$ either satisfies
\begin{equation*}
\int_{ B_\Phi(R)} |\omega|^2dv = o(R^{\lambda}) \qquad as \  R \rightarrow  \infty,
\end{equation*}
or
\begin{equation*}
\int_{ B_\Phi(R)} |\omega|^2dv = o(R^{\overline{\lambda}}) \qquad as \  R \rightarrow  \infty,
\end{equation*}
where $\lambda$ and $\overline{\lambda}$ are as in Theorem 3.1 and Theorem 3.4 respectively,
then $\omega=0$.  In
particular, if $\omega $ has finite energy, then $\omega=0$.
\end{corollary}

By Theorem 3.2 and Theorem 3.5, we derive the following two corollaries:
\begin{corollary}
Let $(M, g)$, $\xi:E \rightarrow M$, $\Phi$, $B_\Phi(R_0)$ and $\omega\in A^p(\xi)$ be as in Theorem 3.2. Set $v=\frac{\nabla\Phi}{|\nabla\Phi|}$.
 If $\omega$ satisfies \\
 $$\frac{|\omega|^2}{2}\geq |i_{v}\omega|^2$$
 on $\partial B_\Phi(R_0)$ and
$$\int_{ B_\Phi(R)\setminus B_\Phi(R_0)} |\omega|^2dv = o(R^{\lambda(R_0)}) \qquad as \qquad R \rightarrow  \infty, $$
then  $\omega\equiv0$ on $M\setminus B_\Phi(R_0)$. In
particular, if $\omega $ has finite energy, then $\omega\equiv0$ on $M\backslash B_\Phi(R_0)$.
\end{corollary}
\begin{corollary}
   Let $(M, g)$, $\xi:E \rightarrow M$, $\Phi$, $B_\Phi(R_0)$ and $\omega\in A^1(\xi)$ be as in Theorem 3.5.
   If
   \begin{equation*}
   \int_{ B_\Phi(R)\setminus B_\Phi(R_0)}|\omega|^2dv = o(R^{\overline{\lambda}(R_0)})  \qquad as \ R \rightarrow  \infty,
   \end{equation*}
   then  $\omega\equiv0$ on $M\setminus B_\Phi(R_0)$. In
particular, if $\omega $ has finite energy, then $\omega\equiv0$ on $M\backslash B_\Phi(R_0)$.
\end{corollary}

\begin{remark}
If $\omega$ in Corollary 3.4 or Corollary 3.5 posses the unique continuation property, then $\omega=0$ on the whole $M$.
\end{remark}
It follows from Theorem 3.6 that
\begin{corollary}
Let $M^m$ be a complex $m$-dimensional Stein manifold with metric and
$\Phi$ as in Theorem 3.6. Suppose $\xi: E \rightarrow M$ is a Riemannian vector bundle and
$\omega \in A^p(\xi)$ with $p < m$ is $J$-invariant. If $\omega$ satisfies the conservation law and
\begin{equation*}
\int_{B_{\Phi}(\rho)}|\omega|^2dv = o(\rho^{2(m-p)}) \ as \ \rho \rightarrow \infty,
\end{equation*}
then $\omega=0$.
 \end{corollary}

From Theorem 3.3 and Theorem 3.7, we get
\begin{corollary}
 Let $M^m$, $\xi$, $K_r$ and $\omega$ be as in Theorem 3.3 (resp. Theorem 3.7). Then $\omega=0$ if $\omega$ satisfies
\begin{equation*}
\int_{B_{\rho}(x_0)}|\omega|^2dv = o(\rho^{\lambda}) \qquad (resp. \ o(\rho^{\overline{\lambda}})) \ as \ \rho \rightarrow \infty,
\end{equation*}
where $\lambda$ (resp. $\overline{\lambda}$) is given by (3.9) (resp. (3.18)).
\end{corollary}

By  Corollary 3.1, we have
\begin{corollary}
Let $M^m$ $(m > 2)$ be a complex $m$-dimensional Stein manifold with metric and $\Phi$ as in Example 2.4.
 If $M$ has constant scalar curvature and
 $$\int_{B_{\Phi}(\rho)}|\Omega_{ric}|^2 dv= o(\rho^{2(m-2)}),$$  then
$M$ is holomorphically isometric to $C^m$.
\end{corollary}
\begin{proof}
 Actually, in Example 2.4, $M$ is embedded in some $C^N$ as a closed complex submanifold and endowed with the induced metric. From the growth condition and Corollary 3.1, we derive that $M$ is Ricci flat and thus its scalar curvature is zero.  It is clear that the second fundamental form of $M$ as a submanifold in $C^N$  vanishes identically.  Therefore $M$ is holomorphically isometric to $C^m$.
\end{proof}

By Corollary 3.2, we yield
\begin{corollary}
Let $M^m$ $(dim_C M =m)$ be a complete K\"ahler manifold with a pole. Assume
that the radial curvature of $M$ satisfies one of the following two conditions:

(i) $-\frac{A}{(1+r^2)^{1+\epsilon}} \leq K_r \leq \frac{B}{(1+r^2)^{1+\epsilon}}$ with $\epsilon>0$, $A \geq 0$, $0 \leq B < 2\epsilon$
and $1+ (2m-1)[1-\frac{B}{2\epsilon}]-4e^{A/2\epsilon} > 0$;

(ii) $-\frac{a^2}{1+r^2} \leq K_r \leq \frac{b^2}{1+r^2}$ with $a\geq 0$, $b^2 \in[0,1/4]$ and
$2+ (2m-1)[1+\sqrt{1-4b^2}]-4[1+\sqrt{1+4a^2}] > 0$.

Then $M$ is Ricci flat if it has constant scalar curvature and satisfies the following growth condition:
$$\int_{B_{\Phi}(\rho)}|\Omega_{ric}|^2 dv = o(\rho^{\overline{\lambda}}),$$
where
\begin{equation*}
\overline{\lambda}= \begin{cases}
1 + (2m-1)(1-\frac{B}{2\epsilon})-4e^{A/2\epsilon}] & \text{if $K_r$ satisfies (i),} \\
1 + \frac{(2m-1)[1+\sqrt{1-4b^2}]}{2}-2[1+\sqrt{1+4a^2}] & \text{if $K_r$ satisfies (ii).}
\end{cases}
\end{equation*}
\end{corollary}

 Some Ricci flatness results were established in \cite{NST} for a compete ALE K\"ahler manifold by assuming nonpositivity or nonnegativity of Ricci curvature and a growth condition of the scalar curvature. On the other hand, if the curvature of a K\"ahler manifold decays like $r^{-2-\epsilon}$, the authors in \cite{MSY,NST} could deduce, under some extra conditions, that $M$ is isometrically biholomorphic to $C^m$. In following, we show that Corollary 3.9 can be used to establish such a kind of results.
\begin{corollary}
Let $M^m$ $(dim_C M=m)$ be a complete K$\ddot{a}$hler manifold whose sectional  curvature dose not change sign. Suppose that $M$ has a pole $x_0$ and
 the radial curvature of $M$ satisfies one of the following  two conditions

(i) $-\frac{a^2}{1+r^2} \leq K_r \leq 0$ with $a\geq 0$  and
$m-[1+\sqrt{1+4a^2}] > 0$;

(ii) $0 \leq K_r \leq \frac{b^2}{1+r^2}$ with  $b^2 \in[0,1/4]$ and
$(2m-1)[1+\sqrt{1-4b^2}]-6> 0$.

Then $M$ is  holomorphically isometric to $C^m$ if it has constant scalar curvature and satisfies the following growth condition:
$$\int_{B_{\Phi}(\rho)}|\Omega_{ric}|^2dv = o(\rho^{\overline{\lambda}}),$$
where
\begin{equation*}
\overline{\lambda}= \begin{cases}
2m-2[1+\sqrt{1+4a^2}]& \text{if $K_r$ satisfies (i),}\\
\frac{(2m-1)[1+\sqrt{1-4b^2}]-6}{2} & \text{if $K_r$ satisfies (ii).}
\end{cases}
\end{equation*}
\end{corollary}
\begin{proof}
As in Corollary 3.9, $M$ is Ricci flat. Since the curvature of $M$ does not change sign, we deduce that $M$ is flat. It follows that $M$ is holomorphically isometric to $C^m$.
\end{proof}

\section{Monotonicity Formulae for volumes }

In this section, we establish monotonicity formulae for volumes, especially for the volumes of minimal submanifolds whose
outer spaces poss some suitable exhaustion functions.

To estimate the volumes, we introduce the following notion. Let $(N^n, g)$ be a Riemannian manifold and let $\Phi$
be a Lipschitz continuous function on $N^n$ satisfying:

(4.1) $\Phi \geq 0$ and $\Phi$ is an exhaustion function of $M$;

(4.2) $\Psi=\Phi^2$ is of class $C^\infty$ and $\Psi$ has only discrete critical points;

(4.3) Fix an integer $m \leq n$, the constant $\widehat{k}_1= \underset{x \in M}{\inf}\{\lambda_1(x)+ \cdots + \lambda_m(x)\} \geq 0$,
where $\lambda_1(x) \leq \lambda_2(x) \leq \cdots \leq \lambda_n(x)$ are the eigenvalues of $Hess(\Psi)$, and the
constant $k_2= \underset{x \in M}{\sup}|\nabla \Phi|^2$ is finite. Set
\begin{equation}
\tag{4.4}
\widehat{\lambda}= \frac{\widehat{k}_1}{2k_2}.
\end{equation}
Notice that (4.3) implies in particular that $\Psi$ is strictly $m$-subharmonic. If $m=n$, then $\Phi$ becomes
a real $0$-exhaustion function in the sense of $\S2$.
\begin{theorem}
Let $(N^n, g)$ be a Riemannian manifold with an exhaustion function satisfying (4.1), (4.2) and (4.3). Let $B_{\Phi}(t) \subset N$
be the sublevel set of $\Phi$. Suppose $i: M^m \rightarrow N^n$ is a minimal submanifold of dimension $m$, then
\begin{equation*}
\frac{Vol(M\cap B_{\Phi}(\rho_1))}{\rho_1^{\widehat{\lambda}}} \leq \frac{Vol(M\cap B_{\Phi}(\rho_2))}{\rho_2^{\widehat{\lambda}}}
\end{equation*}
for any $0 < \rho_1 \leq \rho_2$, where $\widehat{\lambda}$ is given by (4.4).
\end{theorem}
 \begin{proof}
 Let $\rho: M \rightarrow R^+$  be the restriction of $\Phi$ to $M$. Denote by $D$ and $\nabla$
the covariant derivatives of $N^n$ and $M^m$ respectively. By the composition formula (cf. \cite{EL}), we have
\begin{align*}
 \nonumber \nabla d(\Phi^2\circ i)(X,Y)&= (\nabla d\Phi^2)(di(X), di(Y)) + d\Phi^2[(\nabla di)(X,Y)] \\
&= (\nabla d\Psi)(X, Y) + <A(X,Y), \nabla \Psi> \tag{$4.5$}
\end{align*}
for any $X,Y \in TM$, where $A(X,Y)=(D_XY)^\top$.

 Take $\rho$ as an exhaustion function on $M$. By definition of $\widehat{k}_1$ and the minimality of $M$, we get from (4.5) that
\begin{eqnarray*}
{k}_1 &=& \underset{x\in M}{\inf}\{\widetilde{\lambda}_1(x)+ \cdots + \widetilde{\lambda}_m(x)\}\\
&\geq& \underset{x\in M}{\inf}\{\lambda_1(x)+ \cdots + \lambda_m(x)\} \\
&=& \widehat{k}_1,
\end{eqnarray*}
where $\widetilde{\lambda}_1(x) \leq \widetilde{\lambda}_2(x) \leq \cdots \leq \widetilde{\lambda}_n(x)$
are the eigenvalues of $Hess(\rho^2)$.
Thus ${k}_1$ is positive and $\rho$ is a real $0$-exhaustion function.

Now we choose $\omega= 1\in A^0(M\times R)$ which clearly satisfies the conservation law by (2.3).
From the proof of Theorem 3.1, we see that $k_2$ can be any positive number larger than $\sup_{x\in M}|\nabla \rho|^2$
and $k_1$ can be any positive number less than $\underset{x\in M}{\inf}\{\widetilde{\lambda}_1(x)+ \cdots + \widetilde{\lambda}_m(x)\}$.

By definition of  $k_2$, we have
\begin{equation*}
|\nabla \rho| = |\nabla \Phi|\leq k_2.
\end{equation*}
 Hence the exhaustion function $\rho$ has growth order $\widehat{\lambda}$.
Therefore Theorem 3.1 yields
\begin{equation*}
\frac{Vol(M\cap B_{\Phi}(\rho_1))}{\rho_1^{\widehat{\lambda}}} \leq \frac{Vol(M\cap B_{\Phi}(\rho_2))}{\rho_2^{\widehat{\lambda}}}.
\end{equation*}
\end{proof}
 \begin{remark}
 (1) From the proof of Theorem 5.1, we see that if $m-2-mp\delta >0$ (resp. $m-2-2p\delta >0$) for a complete submanifold
 with parallel mean curvature  (resp. for a complete minimal submanifold) in $R^N$, then $\rho^2$ is a real $p$-exhaustion function.
 (2) If we consider the identity immersion $x=I: N \rightarrow N$ in Theorem 4.1, then we get in particular the following
 \begin{equation*}
 \frac{Vol( B_{\Phi}(\rho_1))}{\rho_1^{\widehat{\lambda}}} \leq \frac{Vol( B_{\Phi}(\rho_2))}{\rho_2^{\widehat{\lambda}}}
\end{equation*}
 for any $0 < \rho_1 \leq \rho_2$.
 \end{remark}

\begin{corollary}
Let $(N^n, g)$ be a Riemannian manifold with an exhaustion function satisfying (4.1), (4.2) and (4.3) outside a sublevel set
$B_{\Phi}(R_0)$ for some $R_0 > 0$.
Suppose $M^m \rightarrow N^n$ is a minimal submanifold and $x_0 \in M$.
Let $B_{x_0}(r) \subset N$ be the geodesic ball of radius $r$
centered at $x_0$. Then
$$\frac{Vol(M\cap (B_{x_0}(\rho_1)- B_{x_0}(R_0)))}{\rho_1^{\widehat{\lambda}}}\leq
\frac{Vol(M\cap (B_{x_0}(\rho_2)- B_{x_0}(R_0)))}{\rho_2^{\widehat{\lambda}}}$$
 for any $0 < R_0 \leq \rho_1< \rho_2 $.
\end{corollary}
\begin{proof}
 By the proof of Theorem 3.2 and Theorem 4.1, we can complete the proof of this corollary.
 \end{proof}

When  choosing the distance function to be the exhaustion function, we can get the following lemma of growth order.
\begin{lemma}
Let $(N^n, g)$ be a complete Riemannian manifold with a pole  $x_0$ and let $r$ be the distance function relative to $x_0$.
 Assume that the radial curvature  $K_r$ of $N$ satisfies one of the following conditions:

(i)  $ K_r \leq -\beta^2$ with  $\beta> 0$;

(ii)  $ K_r \leq \frac{B}{(1+r^2)^{1+\epsilon}}$ with $\epsilon>0$,  $0 \leq B < 2\epsilon$;

(iii)  $ K_r \leq \frac{b^2}{1+r^2}$ with  $b^2 \in[0,1/4]$.\\
Then $r^2$ is an exhaustion function satisfying  (4.1), (4.2), (4.3) and
\begin{equation}
\tag{4.6}
\widehat{\lambda}= \begin{cases}
1+(m-1)\beta r\coth(\beta r) & \text{if $K_r$ satisfies (i),} \\
m(1-\frac{B}{2\epsilon}) & \text{if $K_r$ satisfies (ii),} \\
\frac{m(1+\sqrt{1-4b^2})}{2} & \text{if $K_r$ satisfies (iii).}
\end{cases}
\end{equation}

\end{lemma}
\begin{proof}
(i) By the Hessian comparison theorem
\begin{eqnarray*}
Hess(r^2)&=& 2dr \otimes dr + 2rHess(r) \\
&\geq&  2dr \otimes dr + 2r\beta r\coth(\beta r)(g - dr \otimes dr)
\end{eqnarray*}
in the  sense of quadratic forms and $\beta r\coth(\beta r) \geq 1$.
Thus we have $\widehat{k}_1=\underset{x \in M}{\inf}\{\lambda_1(x)+ \cdots + \lambda_m(x)\}= 2+2(m-1)\beta r\coth(\beta r)$ and
$k_2=|\nabla r|^2=1$. Therefore $\widehat{\lambda}=1+(m-1)\beta r\coth(\beta r)$.

The other cases can be proved similarly.
\end{proof}

\begin{theorem}
Let $(N^n, g)$ be a complete Riemannian manifold with a pole $x_0$  and let $r$ be the distance function relative to $x_0$.
Assume that the radial curvature  $K_r$ of $N$ satisfies one of the following conditions:

(i)  $ K_r \leq -\beta^2$ with  $\beta> 0$;

(ii)  $ K_r \leq \frac{B}{(1+r^2)^{1+\epsilon}}$ with $\epsilon>0$,  $0 \leq B < 2\epsilon$;

(iii)  $ K_r \leq \frac{b^2}{1+r^2}$ with  $b^2 \in[0,1/4]$.\\
Suppose $i: M^m \rightarrow N^n$ is a minimal submanifold and $x_0 \in M$. Let $B_{x_0}(r) \subset N$ be the geodesic ball of radius r
centered at $x_0$. Then
$$\frac{Vol(M\cap B_{x_0}(\rho_1))}{\rho_1^{\widehat{\lambda}}} \leq
 \frac{Vol(M\cap B_{x_0}(\rho_2))}{\rho_2^{\widehat{\lambda}}},$$
for any $0 <  \rho_1< \rho_2 $, where $\widehat{\lambda}$ is given by (4.6).
\end{theorem}
\begin{proof}
 Let $\rho : M\rightarrow R^+$ be the extrinsic
distance function of $M$ relative to  $x_0$, that is, $\rho(x)=r\circ i(x)$ for $x\in M$. By Theorem 4.1 and Lemma 4.1, we can immediately
get the theorem.
\end{proof}

When outer space is Euclidean space, we can recapture the classical Monotonicity formulae of
the volume of minimal submanifolds.
\begin{corollary}
(cf. \cite{CM}) Suppose that $M\subset R^n$ is a minimal submanifold and the origin $o\in M$. Let $B(r)$ be the Euclidean ball of radius $r$
centered at $o$, then for all $0 < \rho_1< \rho_2$
$$\frac{Vol(M\cap B(\rho_1))}{\rho_1^m}\leq \frac{Vol(M\cap B(\rho_2))}{\rho_2^m}$$
with equality  if and only if $M$ is an affine $m$-plane.
\end{corollary}
\begin{proof}
By choosing $B=0$ in Theorem 4.2 (ii), the above inequality can be obtained directly.

For the second conclusion,
checking the procedure of the proof of Theorem 3.1, the above equality holds if and only if
$$|\nabla r(x)|=|Dr(x)|=k_2=1 \qquad for\  any \ x\in M,$$
that is
\begin{equation*}
<x,\eta>=0
\end{equation*}
for any $\eta \in T^\bot_xM$. For $x\neq 0$, we defined
\begin{equation*}
e_x=\frac{x}{|x|}\in T_xM.
\end{equation*}
Then we have a unit vector field $e$ on $M\setminus \{o\}$. Obviously the integral curves of $e$ are straight lines in $R^n$
passing through $o$, because $D_ee=0$. We conclude that $M$ is an $m$-plane passing through $o$.
\end{proof}

\section{Application to Bernstein type problem}
In this section, we will investigate submanifolds in a Euclidean space  with parallel mean curvature.
It is known that the Gauss maps of such submanifolds are harmonic. It is natural to apply the method developed in previous
 sections to establish Liouville theorems of the Gauss map, which are equivalent to Bernstein type results of the submanifolds.
 Throughout this section, we shall  make use of the following notations:\\
 \hspace*{0.56cm} $\rho$ = the extrinsic distance relative to a point $o\in M$,\\
 D(R) = the intersection of $M$ with the Euclidean $R$-ball centered at $o\in M$.

\begin{center}
\large{\textbf{5.1. Submanifolds in Euclidean spaces}}
\end{center}

\begin{theorem}
Let $M^m(m\geq 3)$ be a complete Riemannian manifold and let $x: M\rightarrow R^N$ be a proper immersion  with parallel mean curvature.
Assume that $\delta=\underset{x \in M}{\sup}\rho(x)||A||_x$ is finite, where
$A$ is the second fundamental form of $M$.

 (i) If the constant $\mu_1=m-2-m\delta>0$ and
 \begin{equation}
  \int_{D(\rho)}{|A|^2} dv=o(\rho^{\mu_1}) \ \ as\ \rho\rightarrow \infty,
  \end{equation}
  then $M$ is an affine $m$-plane in $R^N$.

  (ii) If $M$ is minimal, the constant $\mu_2=m-2-2\delta>0$ and
   \begin{equation}
  \int_{D(\rho)}{|A|^2} dv=o(\rho^{\mu_2}) \ \ as\ \rho\rightarrow \infty,
  \end{equation}
  then $M$ is an affine $m$-plane in $R^N$.
\end{theorem}
\begin{proof}
 As in the proof of Theorem 4.1, we take $\Phi=\rho$ as an exhaustion function for $M$. We need to prove that $\Phi$ is a real
 $1$-exhaustion function under the assumptions. Let $\lambda_1(x) \leq \lambda_2(x) \leq \cdots \leq \lambda_m(x) $
 be the eigenvalues of $Hess(\rho^2)$. From (4.5), we get
 \begin{equation}
 Hess(\rho^2)(X,X)=2<X,X>+2<A(X,X),x>.
 \end{equation}
 It follows that
 \begin{equation*}
 2-2\delta  \leq \lambda_i \leq 2+2\delta,\qquad i=1,\cdots,m.
 \end{equation*}
 Hence
 \begin{eqnarray}
\nonumber k_1&=&\underset{x\in M}{\inf}\{\underset{i=1}{\overset{m}{\sum}}\lambda_i(x)-2\lambda_m(x)\}\\
              &\geq& 2m-4-2m\delta.
 \end{eqnarray}
 By definition of $\lambda$ and $|\nabla\rho|\leq 1$, we have from (2.11) and (5.4) that
 \begin{equation*}
 \lambda \geq \frac{k_1}{2}= m-2-m\delta.
 \end{equation*}
 Since $M$ has parallel mean curvature, its Gauss map $\gamma: M \rightarrow G_m(R^N)$ is harmonic. Thus $d\gamma$
 satisfies the conservation law, i.e., $divS_\gamma=0$. From \cite{EL}, we know that $|d\gamma|^2=\|A\|^2$.
 By applying Corollary 3.3 to $d\gamma$ under the assumption (5.1), we deduce $d\gamma=0$, that is
  $\gamma$ is constant. Hence $M$ is an $m$-plane.

  Now suppose that $M$ is minimal. By (5.3) and the minimality of $M$, we deduce
  \begin{eqnarray*}
    k_1&=&\underset{x\in M}{\inf}\{\underset{i=1}{\overset{m}{\sum}}\lambda_i(x)-2\lambda_m(x)\}\\
              &\geq& 2m-4-4\delta.
 \end{eqnarray*}
 So $\lambda\geq m-2-2\delta$. Therefore a similar argument yields the result of (ii).
\end{proof}

\begin{remark}
(1) Actually we may get monotonicity formulae for the integral of $\|A\|^2$ in Theorem 5.1.
(2) In the case that $M$ is minimal, Chen \cite{Ch} proved that if $\delta< 1$ (without any energy growth condition),
then $M$ is an $m$-plane in $R^N$. (3) Kasue \cite{Ka} proved that if $\|A\|^2= o(\frac{1}{\rho^{2+\epsilon}})$ (as $\rho\rightarrow \infty$)
and $M^m(m\geq 3)$ has one end, then $M$ is an $m$-plane.
\end{remark}

\begin{corollary}
Let $M^m (m\geq 3)$ be a complete Riemannian manifold and let $x:M^m \rightarrow R^N$ be a proper immersion with parallel
 mean curvature (resp. $H=0$). Let $\delta$, $A$, $\rho$ and $\mu_i$ (i=1,2) be as in Theorem 5.1. If $\mu_1$ (resp. $\mu_2>0$)
 and$$\int_{M}{\|A\|^2} dv<+\infty,$$
 then $M$ is an affine $m$-plane in $R^N$.
\end{corollary}

\begin{corollary}
Let $x: M^m \rightarrow R^{m+1} (m\geq 3)$ be a
  properly  immersed complete hypersurface.
Set $\tilde{\delta}=\underset{x \in M}{\sup}\{\rho(x)\underset{i}\max |\eta_i(x)|\}$, where  $\eta_i(x)$$(i=1,\cdots,m)$ are
the principal curvatures of $M$ at $x$. Suppose $M$ satisfies one of the following two conditions:

(i) $M$ has parallel mean curvature, the constant $\tilde{\mu}_1= m-2-m\tilde{\delta}>0$ and
$$\int_{D(\rho)}{\|A\|^2} dv=o(\rho^{\tilde{\mu}_1}) \qquad as \ \rho\rightarrow \infty,$$

(ii) $M$ is minimal, the constant $\tilde{\mu}_2= m-2-2\tilde{\delta}>0$ and
$$\int_{D(\rho)}{\|A\|^2} dv=o(\rho^{\tilde{\mu}_2}) \qquad as \ \rho\rightarrow \infty,$$
Then $M$ is a hyperplane.
\end{corollary}
\begin{proof}
 Let $\vec{n}$ be the unit normal vector
field of $M$ in $R^{m+1}$. From (5.3), we get
$$\lambda_i(x)=2+2\eta_i(x)<\vec{n},x>.$$
Therefore $$2-2\tilde{\delta}\leq \lambda_i(x)\leq 2+2\tilde{\delta}.$$
So $$k_1=\underset{x\in M}{\inf}\{\underset{i=1}{\overset{m}{\sum}}\lambda_i(x)-2\lambda_m(x)\}\geq 2m-4-2m\tilde{\delta}.$$
In particular, if $M$ is minimal, then $$k_1\geq2m-4-4\tilde{\delta}.$$
The remaining arguments are similar to those in the proof of Theorem 5.1.
\end{proof}

\begin{theorem}
Let $x:M^m \rightarrow R^N $ be a  properly  immersed $m$-dimensional $(m\geq 3)$ complete submanifold.
 Set $\delta(R_0)=\underset{M\setminus B(R_0)}{\sup}\{\rho(x)\|A\|(x)\}$
for some regular value $R_0>0$ of $\rho$ and $\nu=\frac{\nabla\rho}{|\nabla\rho|}$.

(i) Suppose $M$ has parallel mean curvature. Suppose $\|i_{\nu}A\|^2\leq \frac{\|A\|^2}{2}$ on $\partial D(R_0)$ and
$\|A\|(x)\leq \frac{m-2-\epsilon_0}{m\rho(x)}$ for some $\epsilon_0 >0$ and any $\rho(x)\geq R_0$. If
\begin{equation}
\int_{D(\rho)}{|A|^2} dv=o(\rho^{\mu_1}) \ \ as\ \rho\rightarrow \infty,
  \end{equation}
  where $\mu_1 = m-2-m\delta(R_0)$, then $M$ is an $m$-plane in $R^N$.

 (ii) Suppose $M$ is a minimal submanifold. Suppose $\|i_{\nu}A\|^2\leq \frac{\|A\|^2}{2}$ on
 $\partial D(R_0)$ and
$\|A\|(x)\leq \frac{m-2-\epsilon_0}{2\rho(x)}$ for some $\epsilon_0 >0$ and any $\rho(x)\geq R_0$. If
\begin{equation}
\int_{D(\rho)}{|A|^2} dv=o(\rho^{\mu_2}) \ \ as\ \rho\rightarrow \infty,
  \end{equation}
  where $\mu_2 = m-2-2\delta(R_0)$, then $M$ is an $m$-plane in $R^N$.
\end {theorem}
\begin{proof}
First, we assume that $M$ has parallel mean curvature.  Set
\begin{eqnarray*}
k_1(R_0)&=& \underset{M\setminus B(R_0)}{\inf}\{\underset{i=1}{\overset{m}{\sum}}\lambda_i(x)-2 \lambda_m(x)\},\\
k_2(R_0)&=& \underset{M\setminus B(R_0)}{\sup}|\nabla\rho|^2,\\
\lambda(R_0)&=&\frac{k_1(R_0)}{2k_2(R_0)},
\end{eqnarray*}
where $\lambda_1(x) \leq \lambda_2(x) \leq \cdots \leq \lambda_m(x) $
 are the eigenvalues of $Hess(\rho^2)$. From the proof of Theorem 5.1, it is easy to see that $$\lambda(R_0)\geq m-2-m\delta(R_0).$$
 If $\|A\|\leq \frac{m-2-\epsilon_0}{m\rho(x)}$ for some $\epsilon_0 >0$ and any $\rho(x)\geq R_0$, then
 $$m-2-m\delta(R_0)\geq \epsilon_0 >0.$$
 Obviously the assumption $\|i_{\nu}A\|^2\leq \frac{\|A\|^2}{2}$ is equivalent to $\|d\gamma(\nu)\|^2\leq \frac{\|d\gamma\|^2}{2}$,
 where $\gamma: M\rightarrow G(m,N)$ is the Gauss map. We have already known that $\gamma$ is harmonic. We deduce from Corollary 3.4
 that $\gamma$ is constant on $M\setminus B(R_0)$ and thus $\gamma$ is constant on $M$ by the unique continuation property of harmonic maps.
 Hence $M$ is an $m$-plane.
\end{proof}

\begin{corollary}
Let $x:M^m \rightarrow R^N $ be an   $m$-dimensional $(m\geq 3)$ complete submanifold  with parallel mean curvature.
Suppose $|i_{v}A|^2\leq \frac{|A|^2}{2}$, where $v=\frac{\nabla\rho}{|\nabla\rho|}$. If
\begin{equation}
\|A\|(x) =o(\frac{1}{\rho(x)}) \qquad as \ \rho(x)\rightarrow +\infty,
\end{equation}
 and
$$\int_M{|A|^2}dv<+\infty,$$
  then $M$ is an affine m-plane in $ R^N $.
\end{corollary}
\begin{proof}
By (5.7) and  Theorem 1.1 of \cite{BJM}, we know that $x: M^m \rightarrow R^N$ is proper. So the extrinsic distance function $\rho$
   is an exhaustion function. The remaining arguments are as above.
\end{proof}
\begin{remark}
Some submanifolds automatically satisfy the condition
\begin{equation}
|i_{v}A|^2\leq \frac{|A|^2}{2}
\end{equation}
in Theorem 5.2 or Corollary 5.3. Recall that the submanifolds $M$ with the second fundamental form $A$ is called
$isotropic$ at $p \in M$
if $A(e,e)$ is independent of the choice of the unit vector $e \in T_pM$. $M$ is said to be $isotropic$ if $M$ is isotropic at every point.
It is easy to verify that an isotropic  submanifold with dimension$ \geq 2$ satisfies (5.8). In section 5.2, we will investigate  minimal
real K$\ddot{a}$hler submanifolds in $R^N$, which automatically satisfy (5.8) too.
\end{remark}
\begin{definition}
(cf. \cite{An})
The integral $\int_M \|A\|^mdv$ for $m$-dimensional minimal submanifold $M$ in $R^N$ is called the total scalar curvature of $M$.
\end{definition}
\begin{lemma}Let $M^m\rightarrow R^{m+p} (m \geq 3)$ be a  minimally immersed complete submanifold. If $M$ has finite total scalar curvature,
then

(1) $\|A\|_x \leq \frac{C}{\rho^m}$ for  $x\in \partial D(\rho)$;

(2) $\int_M \|A\|^2 dv< \infty$.
\end{lemma}
\begin{proof}(1) See \cite{An,Sc,Ty}.

(2) By using (1), the proof goes almost the same as that of \cite{SZ1}, except that one considers a minimal submanifold instead of a minimal hypersurface.  Since the proof has no essential difference,  we omit the detailed proof about (2) here.
\end{proof}
 \begin{remark}
 In \cite{SZ1}, the authors showed  the finiteness of $\|A\|_{L^2(M)}$ for minimal hypersurfaces with finite total scalar curvature.
\end{remark}
\begin{corollary}
Let $x:M^m \rightarrow R^N (m\geq 3)$ be an $m$-dimensional  complete minimal  submanifold in $R^N$ with finite total scalar curvature.
  If $|i_{v}A|^2\leq \frac{|A|^2}{2}$ holds outside a compact subset of $M$,
 then $M$ is an affine $m$-plane in $R^N$.
\end{corollary}
\begin{proof}
   From Lemma 2.4 of \cite{An}, we know that $x: M^m \rightarrow R^N$ is proper. So the extrinsic distance function $\rho$
   is an exhaustion function. The Corollary 5.4 follows immediately from Lemma 5.1 and Corollary 5.3.
\end{proof}

\begin{corollary}
Let $x:M^m \rightarrow R^{m+1}(m\geq 3)$ be a minimally  immersed complete hypersurface with finite total scalar curvature.
Let $\eta_i(x)$$(i=1,\cdots,m)$ be the principal curvatures of $M$ at $x$. If there exists a compact subset $K$ of $M$
such that
\begin{equation}
\underset{i}{\max}\{\eta_i^2(x)\}\leq \frac{1}{2}\|A\|_x^2
\end{equation}
for any $x\in M\setminus K$, then $M$ is an affine  hyperplane in $R^{m+1}$.
\end{corollary}
\begin{proof}
Clearly the result of this corollary follows immediately from  Corollary 5.4.
\end{proof}

\begin{remark}
(1) Suppose $\mu_2(R_0)$ is as in Theorem 5.2 and $R_0$ is sufficiently large so that $\mu_2(R_0)>0$.
To get the result of Corollary 5.5, we only need to assume (5.8) holds on $\partial D(R_0)$; (2) It is known that the principal
curvatures of the catenoid in $R^{m+1}$ are (cf. \cite{TZ}) $\eta_2=\eta_3=\cdots =\eta_m=\eta$, $\eta_1=-(m-1)\eta$.
Clearly the catenoid does not satisfy the condition (5.8) for $m\geq 3$. This shows that the condition (5.8) (at least
on $\partial D(R_0)$) is necessary for the result. (3) The known nontrivial
examples of minimal submanifolds in $R^N$ with finite total scalar curvature include catenoid (cf. \cite{TZ}),
and the Lagrangian catenoid \cite{Ca} and the special Lagrangian submanifolds constructed by Lawlor \cite{La}.
(4) For a minimal submanifold $M^m \subset R^N$  with
finite total scalar curvature, some authors proved the Bernstein type results under other geometric conditions,
 such as the stability (\cite{SZ1,Wa}), having one end (\cite{An}), monotonicity (\cite{Mo}), etc. (5) In \cite{CL}, the authors proved that if $M$ is a submanifold with parallel mean curvature and
$\int_M \|A\|^m dv < +\infty$, then $M$ is minimal.  Hence the results in Corollary 5.4 and Corollary 5.5 actually hold for submanifolds with parallel mean curvature.

\end{remark}

\begin{center}
{\large{\textbf{5.2. The case of real K\"ahler submanifolds}}}
\end{center}

 Let $M^m$ be a K\"ahler manifold of complex dimension $m$ and let $x:M^m \rightarrow R^N$ an isometric immersion.
 We will call $M$ a real K\"ahler submanifold in $R^N$. Considering the complexified cotangent bundle of $M$,
 with the usual procedure, its second fundamental form $A$ can be split into different components as $A^{(p,q)}$,
 the $(p,q)$ component $(p+q=2, 0 \leq p,q \leq 2)$. A real K\"ahler submanifold $M$ is said to be $(p,q)$-geodesic if
 $A^{(p,q)}=0$.

 Notice that $(1,1)-$geodesic submanifolds have various different names (pluriminimal map, circular immersion)
 in literature (cf. \cite{BEFT,DR}). From \cite{DR}, we know that $(1,1)$-geodesic submanifolds coincide with minimal real
 K\"ahler submanifolds. The simplest examples of pluriminimal submanifolds are  minimal surfaces
 in $R^N$ and  K\"ahler submanifolds in $C^n$. In \cite{APS,DG,He}, the authors established
 Weierstrass representation for minimal real  K\"ahler
 submanifolds. Therefore general minimal real K\"ahler submanifolds  exist in abundance in $R^N$ too,
 besides the previous two special classes.

 For a real K\"ahler submanifold $x: M^m \rightarrow R^N$, we have two Gauss maps: the usual Gauss map $\gamma:
  M\rightarrow G_{2m}(R^N)$ and the complex Gauss map $\gamma^C: M\rightarrow G_{m}(C^N)$. In \cite{RT},
  the authors defined the complex Gauss map
  \begin{equation}
  \gamma^C: M\rightarrow G_{m}(C^N)
  \end{equation}
  by assigning to each point $p \in M$ the parallel translation in $C^N$ of the complex $m$-space $dx(T_pM^{(0,1)})$
  to the origin of $C^N$.

  Choosing a local orthonormal frame $\{e_A\}_{A=1}^{2m}=\{e_1,\cdots ,e_m,Je_1,\cdots,Je_m\}$ around $p$ such that
  $(\nabla_{e_A}e_B)_p=0$, we have local unitary frame fields $\{\eta_j\}_{j=1}^m \in T^{1,0}M$ and
  $\{\eta_{\overline{j}}\}_{j=1}^m \in T^{0,1}M$ given by:
  \begin{equation}
 \eta_j=\frac{1}{\sqrt{2}}(e_j-i Je_j), \ \ \eta_{\overline{j}}=\frac{1}{\sqrt{2}}(e_j+ i Je_j).
 \end{equation}
 The complex Gauss map may be expressed by
 \begin{equation}
 \gamma^C=\eta_{\overline{1}}\wedge \cdots \wedge \eta_{\overline{m}}.
 \end{equation}
 Then we get
 \begin{equation}
 (d\gamma^C)_p(X)=\overset{m}{\underset{k=1}{\sum}}\eta_{\overline{1}}\wedge \cdots \wedge A(X,\eta_{\overline{k}})
 \wedge \cdots \wedge\eta_{\overline{m}}.
 \end{equation}
 \begin{lemma}                                                                                                  
 Suppose $x: M^m \rightarrow R^N$ ($dim_CM = m$) is a minimal real K\"ahler submanifold. Then $\gamma^C$ is anti-holomorphic
 and $\|d\gamma^C\|^2=\|A\|^2$.
 \end{lemma}
 \begin{proof}
 By the  assumption that $M$ is $(1,1)$-geodesic, that is, $A(\eta_j,\eta_{\overline{k}})=0$, we get from (5.13) that
 $$(d\gamma^C)_p(\eta_j)=0,\ \ j=1,\cdots,m.$$
 Therefore the partial differential $\partial \gamma^C: T^{1,0}M \rightarrow T^{1,0}G_m(C^N)$ vanishes.
 This proves the anti-holomorphicity of $\gamma^C$. Clearly (5.13) gives
 $$\|d\gamma^C\|^2=2\overset{m}{\underset{j,k=1}{\sum}}\|A(\eta_{\overline{j}},\eta_{\overline{k}})\|^2=\|A\|^2.$$
 \end{proof}
 \begin{remark}
 The authors in \cite{RT} established the anti-holomorphicity of $\gamma^C$ for minimal real K\"ahler submanifolds.
 Let $H_m(C^N)$ be the space of complex $m$-dimensional isotropic subspaces of $C^N$. Actually the isometric immersion $x$
 factors through $H_m(C^N)$\\
 $\subset G_m(C^N)$ and thus $\gamma^C$ becomes an anti-holomorphic map into $H_m(C^N)$. Therefore their
 result generalized Chern's result \cite{Che} about the Gauss map of minimal surfaces in $R^N$.
 \end{remark}

 \begin{theorem}
 Let $x: M^m \rightarrow R^N (m\geq 2)$ be a  complete minimal real K\"ahler submanifold with complex dimension $m$. If
$$\|A\|(x) =o(\frac{1}{\rho(x)}) \qquad as \ \rho(x)\rightarrow +\infty$$
 and
$$\int_M{|A|^2}dv<+\infty,$$
 then $M$ is an affine $2m$-plane in $R^N$.
 \end{theorem}
 \begin{proof}
 Clearly $\gamma^C$ as a special harmonic map satisfies the conservation law, that is, $divS_{\gamma^C}=0$. Let $\nu$ be as
 in Theorem 5.2 and Corollary 5.4.
  By Lemma 5.2, $\gamma^C$ is anti-holomorphic, so $d\gamma^C$ is pluriconformal. Therefore
 we have
 \begin{eqnarray*}
 \frac{|d\gamma^C|^2}{2} &\geq& (d\gamma^C \odot d\gamma^C)(\nu,\nu)\\
 &=& |i_{\nu}d\gamma^C|^2
 \end{eqnarray*}
 as shown in the proof of Theorem 3.4. Hence the result of Theorem 5.3 follows immediately from Corollary 5.3.
\end{proof}

 By Lemma 5.1 and Theorem 5.3, we can obtain the following theorem.
\begin{theorem}
 Let $x: M^m \rightarrow R^N (m\geq 2)$ be a complete minimal real K\"ahler submanifold with complex dimension $m$. If $M$ has finite
 total scalar curvature, then $M$ is an affine $2m$-plane in $R^N$.
 \end{theorem}
 \begin{corollary}
 (\cite{Mo}) Let $M^m \rightarrow C^n (m\geq 2)$ be a complete K\"ahler submanifold with complex dimension $m$. If $M$ has finite total
 scalar curvature, then $M$ is an affine complex $m$-plane.
 \end{corollary}
 \begin{remark}
 The assumption $m \geq 2$ in Theorem 5.4 or Corollary 5.6 is necessary, because there are many
 nontrivial examples of minimal surfaces in $R^N$ with finite total curvature (cf. \cite{Os}).
 \end{remark}

\vspace{0.5cm}
$\mathbf{Acknowledgements}$.  The authors would also like to thank Dr. Y.B. Han for his careful reading of the manuscript.



\begin{thebibliography}{99}

\bibitem[1]{An} M.T. Anderson, The compactification of a minimal submanifold in Euclidean space by the Gauss map, IHES preprint, (1984).

\bibitem[2]{APS}  C. Arezzo, G.P. Pirola and  M. Solci, The Weierstrass representation for pluriminimal submanifolds,
       arXiv:math/0104124v1 [math.DG], 2001
\bibitem[3]{Ba} P. Baird, Stress-energy tensors and the Lichnerowicz Laplacian, Journal of Geom. and Phys.  58 (2008), 1329-1342.
\bibitem[4]{BE} P. Baird, J. Eells, A conservation law for harmonic maps, Geometry Symposium, Utrecht 1980: Lecture notes in Mathematics,
                   Vol.894, Springer (1982), 1-25.

\bibitem[5]{BEFT} F.E. Burstall, J.H. Eschenburg, M.J. Ferreira and R. Tribuzy, K\"ahler submanifolds
       with parallel pluri-mean curvature, Diff. Geom. Appl.  20(1) (2004),47-66.

\bibitem[6]{BJM} G.P. Bessa, L. Jorge and J.F. Montenegro, Complete submanifolds of $R^n$ with finite topology,  Comm. Annly. Geom.  15(4) (2007), 725-732.

\bibitem[7]{Ca} I. Castro, Minimal Lagrangian submanifolds in complex Euclidean space, Proceedings 1st International Meeting
on Geometry and Topology, (1998), 43-49.

\bibitem[8]{Ch} Q. Chen, Uniqueness of minimal submanifolds in Euclidean Space, Ann. Global Anal. Geom.  16 (1998), 413-418.

\bibitem[9]{Che} S.S. Chern, Minimal surfaces in Euclidean space of N dimensions, Sympos. in Honour of Marston Morse, Princetion
Univ. Press, Princeton, N.J.  (1965), 187-198.

\bibitem[10]{CL}  L.F. Cheung, P.F. Leung, The mean curvature and volume growth of complete noncompact submanifolds, Diff. Geom. Appl. 8 (1998), 251-256.

\bibitem[11]{CM} T.H. Colding, W.P. Minicozzi, Minimal submanifolds, Bull. London Math. Soc.  38 (2006), 353-395.




\bibitem[12]{DG} M. Dajczer,  D. Gromoll, The Weierstrass representation for complete minimal real Kaehler submanifolds of codimension two,
                    Invent. Math.  119 (1995), 235-242.


\bibitem[13]{Do} Y.X. Dong, Monotonicity formulae and holomorphicity of harmonic maps between K\"ahler manifolds,
                     arXiv:1011.6016v2 [math.DG], 2011.


\bibitem[14]{DR} M. Dajczer, L. Rodriguez, Rigidity of real Kaehler submanifolds, Duke Math. J.  53(1) (1986), 211-220.

\bibitem[15]{DW} Y.X. Dong,  S.S. Wei, On vanishing theorems for vector bundle valued p-forms and their applications, Comm. Math. Phys. Vol. 304  (2011), 329-368.

\bibitem[16]{EL}  J. Eells, L. Lemaire, Selected topics in harmonic maps, CBMS Reg. Conf. Ser. Math. 50, Amer. Math. Soc., Providence, 1983.


\bibitem[17]{FZ} H. Fu, Z. Li, The structure of complete manifolds with weighted Poincar$\acute{e}$ inequalities and minimal hypertsurfaces,
International J. Math. 21 (2010), 1-8.

\bibitem[18]{GW} R.E. Greene, H. Wu, Function theory on manifolds which posses a pole, Lecture Notes in Math. Vol.699, Springer-Verlag, 1979.

\bibitem[19]{He} P. Hennes, Weierstrass Representation of minimal real K\"ahler submanifolds, Ph.D. Thesis, state University
  of New York at Stony Brook, (2001).

\bibitem[20]{JXY} J. Jost, Y. L. Xin and Ling Yang, The geometry of Grassmannian manifolds and Bernstein type theorems for higher codimension,
                     arXiv:1109.6394v1, 2011.



\bibitem[21]{Ka} A. Kasue, Gap theorem for minimal submanifolds of Euclidean space, J. Math. Soc. Japan.  38(3) (1986), 473-492.

\bibitem[22]{Kas}  M. Kassi, A Liouville Theorem for F-harmonic maps with finite F-energy, Electronic Journal of differential
   Equations, 15 (2006), 1-9.


\bibitem[23]{La} G. Lawlor, The angle criterion, Invent. Math.  95 (1989), 437-446.

\bibitem[24]{LW} H.Z, Li, G.X. Wei, Stable complete minimal hypersurfaces in $R^4$, Matem$\acute{a}$tica Contempor$\hat{a}$nea, 2005.

\bibitem[25]{Mo} H. Moore, Minimal submanifolds with finite total scalar curvature, Indiana Univ. Math. J.  45 (1996), 1021-1043.

\bibitem[26]{Mor} A. Moroianu, Lectures on K\"ahler Geometry, London Mathematical Society Student Texts 69.

\bibitem[27]{MSY} N. Mok, Y.T. Siu and S.-T. Yau, The Poincar$\acute{e}$-Lelong equation on complete K\"ahler manifolds, Compositio Math. 44 (1981), 183-218.

\bibitem[28]{NST} L. Ni, Y.G. Shi and L.F. Tam, Ricci flatness of asymptotically locally Euclidean metrics, Trans. Amer. Math. Soc. 355 (2002), 1933-1959.

\bibitem[29]{Os} R. Osserman, A Survey of Minimal Surfaces, Dover Pub., 1986.

\bibitem[30]{PRS} S. Pigola, M. Rigoli and  A.G. Setti, Vanishing and finiteness results in geometric analysis:
a generalization of the Bochner technique, Prog. in Math.  Vol.266 (2008).

\bibitem[31]{PV} S. Pigola, G. Veronelli, Remarks on $ L^{p} $-vanishing results in geometric analysis. arXiv:1011.5413v1 [math.DG].



\bibitem[32]{RS}  M. Rigoli, A. G. Setti, Energy estimates and Liouville theorems for harmonic maps,
                  Intern.  J. of Math. Vol. 11, No. 3 (2000) 413-448.
\bibitem[33]{RT} M. Ridoli, R. Tribuzy, The Gauss map for K\"ahlerian submanifolds of $R^n$,
Transactions of the Amer. Math. Soc. 332(2) (1992), 515-528.

\bibitem[34]{Sc} R. Schoen, Uniqueness, symmetry, and embeddedness of minimal surfaces, J. Diff. Geom.  18 (1983), 791-809.

\bibitem[35]{Se} H.C.J. Sealey, The stress-energy tensor and vanishing of $L^2$ harmonic forms, preprint.


\bibitem[36]{SZ1} Y.B. Shen, X.H. Zhu, On the stable complete minimal hypersurfaces in $R^{n+1}$, Amer. J. Math.  120 (1998), 103-116.

\bibitem[37]{SZ2} Y.B. Shen, X.H. Zhu, On complete hypersurfaces with constant mean curvature and
     finite $L^p$-norm curvature in $R^{n+1}$, Acta Math. Sinica 21 (2004), 631-642.

\bibitem[38]{Ta1} K. Takegoshi, A non-existence theorem for pluriharmonic maps of finite energy, Math. Z. 192 (1986), 21-27.

\bibitem[39]{Ta2} K. Takegoshi, Energy estimates and Liouville theorems for harmonic maps, Ann. Scient. $\acute{E}$c. Norm. Sup.23 (1990), 563-592.

\bibitem[40]{Ty} T. Tysk, Finiteness of index and total scalar curvature for minimal hypersurfaces, Proc. Amer. Math. Soc. 105 (1989), 429-435.

\bibitem[41]{TZ} L.F. Tam, D. Zhou, Stability properties for higher dimensional catenoid in $R^{N+1}$,
    Proc. Amer. Math. Soc. 137 (2009), 3451-3461.

\bibitem[42]{Wa} Q.L. Wang, On minimal subamnifolds in an Euclidean space, Math. Nachr. 261-262 (2003), 176-180.

\bibitem[43]{XD} S. Xu, Q. Deng, On complete noncompact submanifolds with constant mean curvature and finite total
   curvature in Euclidean space, Arch. Math. 87 (2006), 60-71.

\bibitem[44]{XG} H.W. Xu, J.R. Gu, A general gap theorem for submanifolds with parallel mean curvature in $R^{n+p}$, Comm. Annly. Geom. 15(1)
         (2007), 175-194.


\bibitem[45]{Xi}  Y.L. Xin, Differential forms, conservation law and monotonicity formula, Scientia Sinica (Ser A) Vol.XXIX (1986), 40-50.

\bibitem[46]{XY1} Y.L. Xin, L. Yang, Curvature estimates for minimal submanifolds of higher codimension, Chin. Anna. Math.
          30(4) (2009), 379-396.

\bibitem[47]{XY2} Y.L. Xin, L. Yang, Convex functions on Grassmannian manifolds and Lawson-Osserman problem, Adv. in Math.
 219(4) (2008),  1298-1326 .




\vspace{1cm}

School of Mathematical Sciences\\
And\\
Laboratory of Mathematics for Nonlinear Science\\
Fudan University, Shanghai 200433\\
P.R. China


yxdong@fudan.edu.cn\\
linhezi$\_$1@163.com









\end{thebibliography}
\end{document}